\documentclass{article}

\title{A general approach to transversal versions of Dirac-type theorems}
\author{Pranshu Gupta\thanks{University of Passau, Faculty of Computer Science and Mathematics, Passau, Germany. \textit{E-mail:} \texttt{pranshu.gupta@uni-passau.de}. This research was conducted while PG was affiliated with the Hamburg University of Technology.}
    \and 
    Fabian Hamann\thanks{Hamburg University of Technology, Institute of Mathematics, Am Schwarzenberg-Campus 3, 21073 Hamburg, Germany.  \textit{E-mail:} \texttt{fabian.hamann@tuhh.de}.}
    \and    
    Alp M\"uyesser\thanks{Department of Mathematics, University College London, London WC1E 6BT, UK. \textit{E-mail:} \{\texttt{alp.muyesser.21|a.sgueglia}\}\texttt{@ucl.ac.uk}. This research was conducted while AS was a PhD student at the London School of Economics.}    
    \and     
    Olaf Parczyk\thanks{Department of Mathematics and Computer Science, Freie Universität Berlin, Arnimallee 3, 14195 Berlin, Germany. \textit{E-mail:} \texttt{parczyk@mi.fu-berlin.de}.}    
    \and 
    Amedeo Sgueglia\footnotemark[3]
}

\usepackage{graphicx}
\usepackage{amssymb}
\usepackage{amsthm}
\usepackage{amsmath}
\usepackage{enumitem}
\usepackage{framed}
\usepackage{soul}
\usepackage{changepage}
\usepackage{comment}
\usepackage{hyperref}
\usepackage[dvipsnames]{xcolor}
\usepackage{svg}
\usepackage[toc,page]{appendix}
\usepackage{multirow}

\usepackage{tikz}
\tikzset{font=\small}
\usetikzlibrary{calc,math}
\usepackage{caption}

\theoremstyle{plain}
\newtheorem{theorem}{Theorem}[section]
\newtheorem{lemma}[theorem]{Lemma}
\newtheorem{proposition}[theorem]{Proposition}

\newtheorem{observation}[theorem]{Observation}

\theoremstyle{definition}
\newtheorem{definition}[theorem]{Definition}

\newtheorem{remark}{Remark}

\newcommand\rbdelta{\delta_{\mathcal{F},d}^{\textrm{rb}}}
\newcommand\eps{\varepsilon}
\newcommand\N{\mathbb{N}}

\newcommand\cP{\mathcal{P}}
\newcommand\cF{\mathcal{F}}
\newcommand\cC{\mathcal{C}}

\newcommand\cH{\mathcal{H}}
\newcommand\cS{\mathcal{S}}
\newcommand\cG{\mathcal{G}}
\newcommand\fG{\textbf{\textup{G}}}
\newcommand\fH{\textbf{\textup{H}}}
\newcommand{\ab}{\textbf{Ab}}
\newcommand{\con}{\textbf{Con}}
\newcommand{\fac}{\textbf{Fac}}
\newcommand{\cA}{\mathcal{A}}
\newcommand\cK{\mathcal{K}}

\begin{document}
\maketitle
\begin{abstract} 
    Given a collection of hypergraphs $\fH=(H_1, \ldots, H_m)$ with the same vertex set, an $m$-edge graph $F\subset \cup_{i\in [m]}H_i$ is a transversal if there is a bijection $\phi:E(F)\to [m]$ such that $e\in E(H_{\phi(e)})$ for each $e\in E(F)$. How large does the minimum degree of each $H_i$ need to be so that $\fH$ necessarily contains a copy of $F$ that is a transversal? Each $H_i$ in the collection could be the same hypergraph, hence the minimum degree of each $H_i$ needs to be large enough to ensure that $F\subseteq H_i$.  Since its general introduction by Joos and Kim~[Bull. Lond. Math. Soc., 2020, 52(3): 498–504], a growing body of work has shown that in many cases this lower bound is tight. In this paper, we give a unified approach to this problem by providing a widely applicable sufficient condition for this lower bound to be asymptotically tight. This is general enough to recover many previous results in the area and obtain novel transversal variants of several classical Dirac-type results for (powers of) Hamilton cycles. For example, we derive that any collection of $rn$ graphs on an $n$-vertex set, each with minimum degree at least $(r/(r+1)+o(1))n$, contains a transversal copy of the $r$-th power of a Hamilton cycle. This can be viewed as a rainbow version of the P\'osa-Seymour conjecture. 
\end{abstract}

\section{Introduction}
This paper concerns the study of transversals over hypergraph collections. Roughly speaking, given a collection of hypergraphs on the same vertex set, we want to select exactly one edge from each hypergraph, so that the selected edges (the transversal) induce a hypergraph with some desired property. 
More precisely, we say that $\fH=(H_1,\ldots,H_m)$ is a \textit{hypergraph collection on vertex set $V$} if, for each $i\in [m]$, the hypergraph $H_i$ has vertex set $V$. We call the collection a \textit{graph collection} if each hypergraph in the collection has uniformity two. Given an $m$-edge hypergraph $F$ on $V$, we say that $\fH$ has a \emph{transversal} copy of $F$ if there is a bijection $\phi\colon E(F)\to [m]$ such that $e\in H_{\phi(e)}$ for each $e\in E(F)$.
We will also use the adjective \textit{rainbow} for a transversal copy of $F$. Indeed, we can think of the edges of hypergraph $H_i$ to be coloured with colour $i$ and, in this framework, a transversal copy of $F$ is a copy of $F$ in $\bigcup_{i \in [m]} H_i$ with edges of pairwise distinct colours. We are interested in the following general question formulated originally by Joos and Kim \cite{joos2020rainbow}.
\begin{equation*}
    \text{\parbox{.92\textwidth}{
    \textbf{Question 1. }Let $F$ be an $m$-edge hypergraph with vertex set $V$, $\cH$ be a family of hypergraphs, and $\fH=(H_1,\dots,H_m)$ be a hypergraph collection on vertex set  $V$ with $H_i \in \cH$ for each $i \in [m]$. 
    Which conditions on $\cH$ guarantee a transversal copy of $F$ in $\fH$?}}
\end{equation*}

By taking $H_1=H_2=\dots=H_m$, it is clear that such a property needs to guarantee that each hypergraph in $\cH$ contains $F$ as a subhypergraph. However, this alone is not always sufficient, not even asymptotically.
For example, Aharoni, DeVos, de la Maza, Montejano and \v{S}\'{a}mal~\cite{rainbowMantel} showed that if $\fG=(G_1,G_2,G_3)$ is a graph collection on $[n]$ with $e(G_i)>(\frac{26-2\sqrt{7}}{81})n^2$ for each $i\in [3]$, then $\fG$ contains a transversal which is a triangle. As shown in~\cite{rainbowMantel}, the constant $\frac{26-2\sqrt{7}}{81}>1/4$ is optimal. On the other hand, Mantel's theorem states that any graph with at least $n^2/4$ edges must contain a triangle.
\par Instead of a lower bound on the total number of edges, it is also natural to investigate what can be guaranteed with a lower bound on the minimum degree. It turns out that even in this more restrictive setting, there can be a discrepancy between the uncoloured and the rainbow versions of the problem. To make this more precise, we give the following two definitions, where, for a $k$-uniform hypergraph $H$ and $1 \le d < k$, we let $\delta_d(H)$ be the minimum number of edges of $H$ that any set of $d$ vertices of $V(H)$ is contained in.
Moreover, for a hypergraph collection $\fH = (H_1,\dots,H_m)$, we denote $|\fH|=m$ and $\delta_d(\fH)=\min_{i\in [m]}\delta_d(H_i)$.

\begin{definition}[Uncoloured minimum degree threshold]
\label{def:min_degree_threshold}
Let $\mathcal{F}$ be an infinite family of $k$-uniform hypergraphs. By $\delta_{\mathcal{F},d}$ we denote, if it exists, the smallest real number $\delta$ such that for all $\alpha>0$ and for all but finitely many $F\in \mathcal{F}$ the following holds.
Let $n=|V(F)|$ and $H$ be any $n$-vertex $k$-uniform hypergraph with $\delta_d(H) \ge (\delta+\alpha)n^{k-d}$.
Then $H$ contains a copy of $F$.
\end{definition}

For example, if $\mathcal{F}$ is the family of graphs consisting of a cycle on $n$ vertices for each $n\in\mathbb{N}$, then we have $\delta_{\cF,1} = 1/2$. Indeed, this follows from Dirac's theorem which states that any graph with minimum degree at least $n/2$ has a Hamilton cycle. 

\begin{definition}[Rainbow minimum degree threshold]

Let $\mathcal{F}$ be an infinite family of $k$-uniform hypergraphs. By $\rbdelta$ we denote, if it exists, the smallest real number $\delta$ such that for all $\alpha>0$ and for all but finitely many $F\in \mathcal{F}$ the following holds.
Let $n=|V(F)|$ and $\fH$ be any $k$-uniform hypergraph collection on $n$ vertices with $|\fH|=|E(F)|$ and $\delta_d(\fH) \geq (\delta+\alpha)n^{k-d}$.
Then $\fH$ contains a transversal copy of $F$.
\end{definition}

If the two values are well-defined, it must be that $\rbdelta\geq \delta_{\mathcal{F},d}$. Indeed, if $H$ contains no copy of $F$, the collection $\fH$ consisting of $|E(F)|$ copies of $H$ does not contain a transversal copy of $H$ either. However, Montgomery, M\"uyesser, and Pehova~\cite{alp_yani_richard} made the following observation which shows that $\rbdelta$ can be much larger than $\delta_{\mathcal{F},d}$. Set $\mathcal{F}=\{k\times (K_{2,3} \cup C_4)\colon k\in \mathbb{N}\}$ where $k\times G$ denotes the graph obtained by taking $k$ vertex-disjoint copies of $G$. It follows from a result of K\"uhn and Osthus \cite{kuhnosthus} that $\delta_{\mathcal{F},1} = 4/9$.
Consider the graph collection $\fH=(H_1,\dots,H_m)$ on $V$ obtained in the following way.
Partition $V$ into two almost equal vertex subsets, say $A$ and $B$, and suppose that $H_1=H_2=\cdots=H_{m-1}$ are all disjoint unions of a clique on $A$ and a clique on $B$. Suppose that $H_m$ is a complete bipartite graph between $A$ and $B$. Observe that each $H_i$ in this resulting graph collection has minimum degree $\lfloor |V|/2 \rfloor$. Further observe that if $\fH$ contains a transversal copy of some $F\in \mathcal{F}$, the edge of $K_{2,3}$ or $C_4$ that gets copied to an edge of $H_m$ would be a bridge (an edge whose removal disconnects the graph) of $F$. However, neither $K_{2,3}$ nor $C_4$ contains a bridge. Hence, $\rbdelta\geq 1/2$.
\par On the other hand, there are many natural instances where $\rbdelta= \delta_{\mathcal{F},d}$. When this equality holds, we say that the corresponding family $\mathcal{F}$ is $d$-\textit{colour-blind}, or just \textit{colour-blind} in the case $\cF$ is a family of graphs (and $d=1$).
For example, Joos and Kim~\cite{joos2020rainbow}, improving a result of Cheng, Wang, and Zhao~\cite{cheng} and confirming a conjecture of Aharoni~\cite{rainbowMantel}, showed that, if $n\geq 3$, then any $n$-vertex graph collection $\fG=(G_1,\ldots,G_n)$ with $\delta(G_i)\geq n/2$ for each $i \in [n]$ has a transversal copy of a Hamilton cycle. This generalises Dirac's classical theorem and implies that the family $\mathcal{F}$ of $n$-cycles is colour-blind\footnote{In fact, in this particular case, the corresponding thresholds are \textit{exactly} the same, and there is no need for an error term. We discuss this aspect of the problem further in our concluding remarks.}.
There are many more families of colour-blind (hyper)graphs. In particular, matchings~\cite{cheng2021rainbow,Lu_codegree_rb-matching_tight,lu2022rainbow,rainbow_matching}, Hamilton $\ell$-cycles~\cite{cheng2021hamilton}, factors~\cite{cheng2021rainbow,alp_yani_richard}, and spanning trees \cite{alp_yani_richard} have been extensively studied.
We recall that for $1 \le \ell <k$, a
$k$-uniform hypergraph is called an $\ell$-cycle if its vertices can be ordered cyclically such that each of its
edges consists of $k$ consecutive vertices and every two consecutive edges (in the natural
order of the edges) share exactly $\ell$ vertices.
In particular, $(k-1)$-cycles and $1$-cycles are known as \textit{tight} cycles and \textit{loose} cycles respectively.
\par In this paper, building on techniques introduced by Montgomery, Müyesser, and Pehova \cite{alp_yani_richard}, we give a widely applicable sufficient condition for a family of hypergraphs $\mathcal{F}$ to be colour-blind. This gives a unified proof of several of the aforementioned results, as well as many new rainbow Dirac-type results. The following theorem lists the applications we derive in the current paper, though we believe that our setting can capture even more families of hypergraphs. 
\begin{theorem}
\label{thm:applications}
    The following families of hypergraphs are all $d$-colour-blind. 
    \begin{enumerate}[label=(\Alph*)]
        \item 
        \label{thm:app_HC2}
        The family of the $r$-th powers of Hamilton cycles for fixed $r \ge 2$ (and $d=1$).
        \item \label{thm:app2}
        \label{thm:app_cycle}
        The family of $k$-uniform Hamilton $\ell$-cycles for the following ranges of $k$, $\ell$, and $d$.
        \begin{enumerate}[label=(B\arabic*)]
            \item $1 < \ell < k/2$ and $d=k-2$;
            \item \label{thm:app_2} $1 \le \ell < k/2$ or $\ell = k-1$, and $d=k-1$;
            \item $\ell=k/2$ and $k/2 < d \le k-1$ with $k$ even.
        \end{enumerate}
    \end{enumerate}
\end{theorem}
\begin{remark} Theorem~\ref{thm:applications}~\ref{thm:app_2} when $\ell=k-1$ was originally proven by Cheng, Han, Wang, Wang, and Yang~\cite{cheng2021hamilton}, who raised the problem of obtaining the rainbow minimum degree threshold for a wider range of $\ell\in[k-2]$. Moreover, the case of Hamilton cycles in graphs (i.e. $k=2$ and $d=\ell=1$) was previously proven by Cheng, Wang, and Zhao \cite{cheng} (and their result was sharpened by Joos and Kim \cite{joos2020rainbow}).
\end{remark}
 
 Theorem~\ref{thm:applications} is derived from our main theorem, Theorem~\ref{thm:main}, in Section~\ref{sec:applications}. The precise statement of Theorem~\ref{thm:main} is quite technical, and will be given in Section~\ref{sec:statement}. In preparation for the formal statement, we now give some intuition for our approach. Firstly, Theorem~\ref{thm:main} is concerned with hypergraph families $\mathcal{F}$ with a `cyclic' structure. That is, we assume there exists a hypergraph $\cA$ such that all $F\in \mathcal{F}$ can be obtained by gluing several copies of $\cA$ in a Hamilton cycle fashion (see Definition~\ref{def:A-cycle}). For example, for $k$-uniform Hamilton cycles, $\cA$ would be a single $k$-uniform edge (see Figure~\ref{fig:edge}), whereas for the $r$-th power of a Hamilton cycle, $\cA$ would be a a clique on $r$ vertices (see Figure~\ref{fig:square}). In the uncoloured setting, most of the well-studied problems fit into this framework, including everything listed in Theorem~\ref{thm:applications}.
 \par A common framework for embedding such hypergraphs with cyclic structure is the absorption method. Suppose one wishes to reprove Dirac's theorem (any $n$-vertex graph $G$ with minimum degree at least $n/2$ contains a Hamilton cycle) via the most common-place variant of the absorption method. Then, the key steps would roughly be as follows. 
 \begin{itemize}[leftmargin=1.5cm]
     \item[Step 1.] \textbf{Set aside a vertex reservoir. } Let $R\subseteq V(G)$ be a small subset so that each vertex $v$ has about $|R|/2$ many neighbours in $R$. A randomly sampled set $R$ satisfies this property with high probability. 
     \item[Step 2.] \textbf{Find an absorber. } Show that $V(G)\setminus R$ contains a small subset $A$ and vertices $v,w\in A$ so that for any small enough subset $L\subseteq V(G)\setminus A$, we have that $G[L\cup A]$ contains a spanning path with endpoints $v$ and $w$.
     \item[Step 3.] \textbf{Cover the remainder via long paths. } Show that all but few vertices of $V(G)\setminus (A \cup R)$ can be covered by pairwise disjoint long paths in $G$.
     \item[Step 4.] \textbf{Build a long path.} Using the property of $R$, connect corresponding endpoints via short paths to build a path $P$ with endpoints $v$ and $w$, vertex-disjoint with $A\setminus\{v,w\}$, and covering all but a few vertices of $V(G)\setminus A$.
     \item[Step 5.] \textbf{Use the absorber.} By the property of $A$, the set $L=V(G)\setminus (A \cup V(P))$ can be used together with $A$ to connect $v$ and $w$ via a path $P'$. Then $P\cup P'$ is the desired Hamilton cycle.
 \end{itemize}
 
 Our main theorem, Theorem~\ref{thm:main}, essentially states that if there is a proof as above that $\delta$ is the uncoloured minimum degree threshold for some $\mathcal{F}$ with cyclic structure, then the rainbow minimum degree threshold of $\mathcal{F}$ is equal to $\delta$. Some partial progress towards such an abstract statement was already made in \cite{alp_yani_richard}. In fact, in \cite{alp_yani_richard}, it was remarked that, plausibly, transversal versions of other Dirac-type results can be shown by adapting the method from \cite{alp_yani_richard}, provided that one can prove strengthened versions of the non-transversal (uncoloured) embedding problem. For example, in \cite{alp_yani_richard}, this strengthening took the form of embedding trees with the location of a single vertex being specified adversarially (see Theorem 4.4 in \cite{alp_yani_richard}). Using such a strengthening, one can translate Steps 1 to 5 above into a coloured setting. On the other hand, the main advantage of our main theorem, Theorem~\ref{thm:main}, is that it does not require to make ad-hoc strengthenings to the uncoloured version of the result, allowing for a very short proof of Theorem~\ref{thm:applications}. To achieve this, we codify, through what we call properties \textbf{Ab} and \textbf{Con}, what it means for there to be streamlined absorption proof for the uncoloured result, and we use the existence of such a proof as a black-box. In our applications, to ensure that the relevant properties hold, we rely on existing lemmas in the literature without having to do any extra work (see Table~\ref{fig:table_applications}). 
 \par In addition to properties \textbf{Ab} and \textbf{Con} which guarantee we can rely on a streamlined absorption proof for the uncoloured result, we have one more hypothesis in Theorem~\ref{thm:main}, which we  call property \textbf{Fac}. One reason why transversal versions of Dirac-type results are more difficult is that every single hypergraph in the collection as well as every single vertex of the host graph needs to be utilised in the target spanning structure (the transversal). This is crucial as demonstrated by the construction given after Definition 1.2. In this construction, the possibility of finding a transversal copy of $\mathcal{F}$ is ruled out by showing that a particular graph in the collection (namely the hypergraph $H_m$) cannot be used in a transversal copy of a $K_{2,3}$ or $C_4$. Therefore, in addition to some properties which are related to the uncoloured case and where colours do not play any role, we require a property concerning the coloured case which we call \textbf{Fac}. This roughly states that vertex-disjoint copies of $\mathcal{A}$ (the building block of the hypergraph we are trying to find) can be found in a rainbow fashion using a fixed, adversarially specified set of hypergraphs from the collection. This ensures that we never get stuck while trying to use up every single colour/hypergraph that we start with. When $\mathcal{A}$ is just a single edge (as it will be the case for Theorem~\ref{thm:applications}~\ref{thm:app_cycle}), the property \textbf{Fac} is essentially trivial to check (see Observation~\ref{obs:fac}). For powers of Hamilton cycles, however, this property is more delicate and, to verify \textbf{Fac}, we rely on a non-trivial coloured property from \cite{alp_yani_richard}.
 
  \begin{remark} In principle, using Theorem~\ref{thm:main}, one can obtain short proofs of many other transversal Dirac-type results.
  But there does exist a natural instance of an uncoloured Dirac-type problem ($3$-uniform Hamilton $2$-cycles with $d=1$) with a proof based on the absorption method \cite{reiher2019minimum}, yet the structure of this proof does not fit into our framework (essentially due to complications with Step $4$). We discuss this further in Section~\ref{sec:concluding_remarks}.
  \end{remark}

\textbf{Organisation.} 
The rest of the paper is organised as follows.
In Section~\ref{sec:statement}, we introduce the necessary terminology for our main Theorem~\ref{thm:main} and give the formal statement.
In Section~\ref{sec:overview}, we introduce some tools that we will use later and provide an overview of the proof method. 
Section~\ref{sec:main_proof} contains the proof of Theorem~\ref{thm:main}, while Section~\ref{sec:applications} deals with its applications.
Finally, some concluding remarks and directions for future research are given in Section~\ref{sec:concluding_remarks}.

\textbf{Notation.} We make a few clarifying points about the notation that we use.
Recall that for a $k$-uniform hypergraph $H$ and $1 \le d < k$, we let $\delta_d(H)$ be the minimum number of edges of $H$ that any set of $d$ vertices of $V(H)$ is contained in.
Moreover, given vertex subsets $S$ and $V$ of a hypergraph, $\deg(S,V)$ denotes the number of $V'\subseteq V$ such that $S\cup V'$ is an edge of the hypergraph.

Recall as well that a hypergraph collection $\fH=(H_1,\ldots,H_m)$ is a collection of (not necessarily distinct) hypergraphs $H_i$, $i\in [m]$, which all have the same vertex set, and $\delta_d(\fH)=\min_{i\in [m]}\delta_d(H_i)$. Given a hypergraph collection $\fH=(H_1,\ldots,H_m)$ with vertex set $V$, and a set $U\subset V$, the collection of graphs $H_i[U]$, $i\in [m]$ induced on the vertex set $U$ is denoted by $\fH[U]$. We set $|\fH|$ to denote the size of $\fH=(H_1,\ldots,H_m)$, so that, in this case, $|\fH|=m$.

As mentioned earlier in the introduction, we will think of the edges of different hypergraphs in a collection as having different colours. In particular, given a hypergraph collection $\fH=(H_1,\ldots,H_m)$, we consider each hypergraph $H_i$ to have a colour $i$. Given a subgraph $H\subset \cup_{i\in [m]}H_i$, edge $e\in E(H)$ can be assigned colour $i$ if $e\in E(H_i)$. $H$ is a transversal if each edge can be assigned a distinct colour. Hypergraphs where an edge colour does not repeat is called \emph{rainbow}. When we say $H\subset \cup_{i\in [m]}H_i$ is \textit{uncoloured}, we emphasise that a colouring has not yet been assigned.  

We use standard hierarchical notation for constants, writing $x\ll y$ to mean that there is a fixed positive non-decreasing function $f$ on $(0,1]$ such that the subsequent statements hold for $x\leq f(y)$. Where multiple constants appear in a hierarchy, they are chosen from right to left. We omit rounding signs where they are not crucial.

\section{Statement of the main theorem}
This section introduces the relevant terminology and states our main Theorem~\ref{thm:main}. 
\label{sec:statement}
\subsection{Links, chains, and cycles}
An \textit{ordered hypergraph} is a hypergraph equipped with a linear order of its vertex set. For convenience, we often index the vertices of an ordered $n$-vertex hypergraph with $\{1,2,\ldots, n\}$ so that $v_i<v_j$ if and only if $i<j$. A subgraph of an ordered hypergraph inherits an ordering from the parent hypergraph in the obvious way. Whenever we state that two ordered hypergraphs are isomorphic, we mean that they are isomorphic as ordered hypergraphs.

\begin{definition}[$\ell$-link]
	Let $k,\ell,m \in \mathbb{N}$ with $\ell \leq  m$. Let $\mathcal{A} = (V,E)$ be an ordered $k$-uniform hypergraph on $m$ vertices. We call $\mathcal{A}$ an $\ell$-link of uniformity $k$ if $\mathcal{A}_s$ and $\mathcal{A}_t$ are isomorphic, where $\mathcal{A}_s = \mathcal{A}[\{v_1,\cdots,v_\ell\}]$, and $\mathcal{A}_t = \mathcal{A}[\{v_{m-\ell + 1},\cdots,v_m\}]$. 
	We refer to $m$ as the order of $\cA$ and we call the ordered hypergraphs $\mathcal{A}_s$ and $\mathcal{A}_t$ the start and the end of $\mathcal{A}$, respectively. 
\end{definition}

\begin{definition}[$\mathcal{A}$-chain]
	Let $k,\ell,m \in \mathbb{N}$ with $\ell \le m$, and $\mathcal{A}$ be an $\ell$-link of uniformity $k$ and order $m$. We say that an ordered hypergraph $\mathcal{P}$ is an $\mathcal{A}$-chain if the following properties hold.
	\begin{enumerate}
	    \item $v(\mathcal{P})=n=(m-\ell)t+\ell$ for some $t \in \mathbb{N}$.
	    \item Set $S_1=\{1,\ldots, m\}$ and for $1 < q \leq t$ define $S_q \subseteq [n]$ recursively as follows.
	    For $1 < q \le t$, if $S_{q-1}=\{s,\ldots, s+m-1\}$, define $S_q=\{s+m-\ell, \ldots, s+2m-\ell-1\}$. Then, for each $1 \leq q \leq t$, the hypergraph $\mathcal{P}_q=\mathcal{P}[\{v_i\colon i\in S_q\}]$ is isomorphic to $\mathcal{A}$.
	    \item Each edge of $\mathcal{P}$ is contained in $\mathcal{P}_q$ for some $q\in [t]$. 
	\end{enumerate}
	We refer to $t$ as the length of the $\cA$-chain and we call $\mathcal{P}_1$ and $\mathcal{P}_t$ the first and the last links of $\mathcal{P}$, respectively. Moreover, we call the start of $\mathcal{P}_1$ and the end of $\mathcal{P}_t$ the start and the end of $\mathcal{P}$, respectively, and refer to them collectively as the ends of $\mathcal{P}$.
\end{definition}

\begin{definition}[$\mathcal{A}$-cycle]
\label{def:A-cycle}
	Let $\mathcal{P}$ be an $\mathcal{A}$-chain. Let $\mathcal{S}$ and $\mathcal{T}$ be the start and the end of $\mathcal{P}$, respectively. Let $\phi$ be the isomorphism between the ordered hypergraphs $\mathcal{S}$ and $\mathcal{T}$, and identify $x\in \mathcal{S}$ with $\phi(x)\in \mathcal{T}$ for each $x \in \mathcal{S}$. We call the resulting (unordered) hypergraph an $\mathcal{A}$-cycle.
\end{definition}

We remark that with $\cA$ being an $\ell$-link of order $m$, if $\cP$ is an $\cA$-chain and $\cC$ is an $\cA$-cycle, then the following holds: $v(\cC) \in (m-\ell)\mathbb{N}$ and $e(\cC)=\frac{e(\cA)-e(\cA_s)}{m-\ell}v(\cC)$, while $v(\cP) \in (m-\ell)\mathbb{N}+\ell$ and $e(\cP) =  \frac{e(\cA)-e(\cA_s)}{m - \ell}v(\cP)-O(1)$, where $O(1)$ stands for a constant which only depends on $\cA$.

Observe that for each $1 \le \ell \le k$, a single $k$-uniform edge induces an $\ell$-link of uniformity $k$ and order $k$, and its chain (resp. cycle) corresponds to a $k$-uniform $\ell$-path (resp. cycle). 
Figure~\ref{fig:edge} shows the case $k=5$ and $\ell=2$.
Similarly, the compete graph on $r$ vertices induces a $(r-1)$-link of uniformity $2$ and order $r$, and its chain (resp. cycle) corresponds to the $(r-1)$-th power of a path (resp. cycle). 
Figure~\ref{fig:square} illustrates the case $r=3$.
Finally, Figure~\ref{fig:pillar} shows that a pillar can also be obtained as an $\cA$-chain.

\begin{figure}[htpb]
    \begin{center}
	
	\begin{tikzpicture}[scale=0.4,
		point/.style ={circle,draw=none,fill=#1,
			inner sep=0pt, minimum width=0.22cm,node contents={}}
		]
		
		\foreach \s in {0,12,24}{
			\foreach \t in {1,2,3,4,5}{
				\node () at ({\s+2*\t-6},0)   [point=black, label=above:${\t}$];
			}
			\draw[line width=0.3mm] (\s,0) ellipse (5cm and 2.45cm);
		}

		\foreach \s in {6,18}{
			\foreach \t in {1,2,3,4,5}{
				\node () at ({\s+2*\t-6},0)   [point=black, label=below:${\t}$];
			}
			\draw[line width=0.3mm] (\s,0) ellipse (5cm and 2.45cm);
		}

	\end{tikzpicture}
	\end{center}
	\captionsetup{font=footnotesize}
	\captionof{figure}{A $5$-uniform $2$-path is an $\cA$-chain, with $\cA$ being (any ordering of) a single $5$-uniform edge. The numbering of the vertices in each edge denotes the (ordered) isomorphism between that edge and $\cA$.}
	\label{fig:edge}
\end{figure}
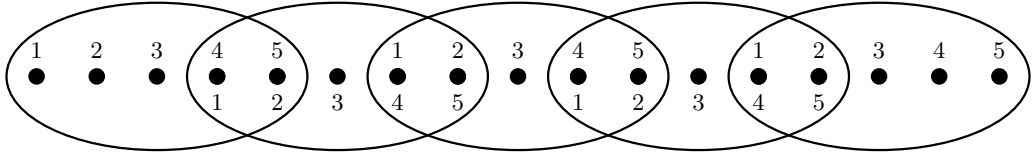

\begin{figure}[htpb]
	\begin{center}
	\begin{tikzpicture}[scale=1,
		point/.style ={circle,draw=none,fill=#1,
			inner sep=0pt, minimum width=0.2cm,node contents={}}
		]
		
		\foreach \s in {0,2,4,6,8}{
			\node (a\s) at ({\s-1},0) [point=black];
			\draw ({sqrt(3)*0.25+\s-1},0.25) node {$1$};
			\node (b\s) at ({\s+1},0) [point=black];
			\draw ({-sqrt(3)*0.25+\s+1},0.25) node {$3$};
			\node (c\s) at (\s,{sqrt(3)}) [point=black, label=below:$2$];
			\draw[line width=0.3mm] (a\s) -- (b\s) -- (c\s) -- (a\s);		
		}

		\foreach \s in {1,3,5,7}{
			\node (a\s) at ({\s-1},{sqrt(3)}) [point=black];
			\draw ({sqrt(3)*0.25+\s-1},{sqrt(3)-0.25}) node {$1$};
			\node (b\s) at ({\s+1},{sqrt(3)}) [point=black];
			\draw ({-sqrt(3)*0.25+\s+1},{sqrt(3)-0.25}) node {$3$};
			\node (c\s) at (\s,0) [point=black, label=above:$2$];
			\draw[line width=0.3mm] (a\s) -- (b\s);		
		}

	\end{tikzpicture}
	\end{center}
	\captionsetup{font=footnotesize}
	\captionof{figure}{The square of a path is an $\cA$-chain, with $\cA$ being (any ordering of) a triangle.}
	\label{fig:square}
\end{figure}
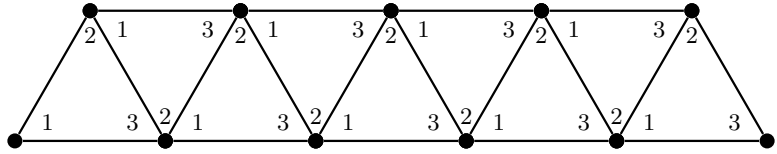

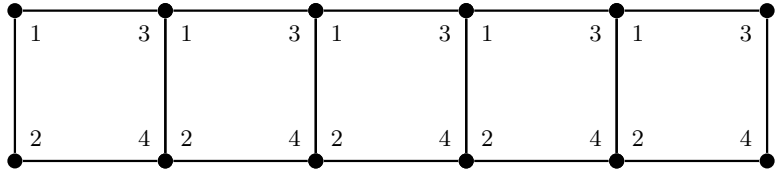
\begin{figure}[htpb]
	\begin{center}
	\begin{tikzpicture}[scale=1,
		point/.style ={circle,draw=none,fill=#1,
			inner sep=0pt, minimum width=0.2cm,node contents={}}
		]
		
		\foreach \s in {0,2,4,6,8}{
			\node (a\s) at (\s,0) [point=black,label=above right:$2$];
			\node (b\s) at (\s+2,0) [point=black,label=above left:$4$];
			\node (c\s) at (\s+2,2) [point=black,label=below left:$3$];
			\node (d\s) at (\s,2) [point=black,label=below right:$1$];
			\draw[line width=0.3mm] (a\s) -- (b\s) -- (c\s) -- (d\s) -- (a\s);
		}
	\end{tikzpicture}
	\end{center}
	\captionsetup{font=footnotesize}
	\captionof{figure}{A pillar is an $\cA$-chain, with $\cA$ being the above ordering of a cycle on $4$ vertices.}
	\label{fig:pillar}
\end{figure}

We now state the properties we require from the link $\cA$ for our main theorem to hold.
\begin{definition}
\label{def:good}
    Let $k,\ell,m \in \mathbb{N}$ with $\ell \le m$, $\mathcal{A}$ be an $\ell$-link of order $m$ and uniformity $k$, and $d\in [k-1]$.
    We say that $\cA$ is $(\delta,d)$-\emph{good} if the following three properties hold.
    \begin{enumerate}[align=right]
    \item[\textbf{$\ab$}]
        For any $\alpha>0$, there exist $0 < \tau, \eta \le \alpha$ and $n_0 \in \mathbb{N}$ so that if $\mathcal{H}$ is a $k$-uniform hypergraph on $n\geq n_0$ vertices with $\delta_d(\mathcal{H}) \geq (\delta + \alpha)n^{k-d}$, then there exists $A\subseteq V(\mathcal{H})$ of size at most $\tau n$ with the following property. 
        
        For any $L\subseteq V(\mathcal{H})\setminus A$ of size at most $\eta n$ with $|L|\in (m-\ell)\mathbb{N}$, there exists an embedding of an $\mathcal{A}$-chain to $\mathcal{H}$ with vertex set $A\cup L$. Furthermore, the embedding of the start and the end of the $\mathcal{A}$-chain does not depend on the subset $L$.
        
    \item[\textbf{$\con$}]
        For any $\alpha>0$, there exist a positive integer $c$ and $n_0 \in \mathbb{N}$ so that if $\mathcal{H}$ is a $k$-uniform hypergraph $\mathcal{H}$ on $n\geq n_0$ vertices with $\delta_d(\mathcal{H}) \geq (\delta + \alpha)n^{k-d}$, the following holds.
            
        Let $\mathcal{S}$ and $\mathcal{T}$ be vertex-disjoint copies of $\cA_s$ in $\mathcal{H}$. Then, $\mathcal{H}$ contains an embedding of an $\mathcal{A}$-chain of length at most $c$ with start $\mathcal{S}$ and end $\mathcal{T}$.
        
    \item[\textbf{$\fac$}]
        For any $\alpha>0$, there exist $\beta_0>0$ and $n_0 \in \mathbb{N}$ so that the following holds for any $n \ge n_0$ and $\beta\leq \beta_0$.
        
        Let $\mathbf{H}$ be a hypergraph collection on vertex set $[n]$ with $|\mathbf{H}| \le \beta n$ and $\delta_d(\fH) \geq (\delta + \alpha)n^{k-d}$.
        Moreover, suppose $e(\mathcal{A})$ divides $|\mathbf{H}|$. Then $\mathbf{H}$ contains a transversal which consists of $|\mathbf{H}|/e(\mathcal{A})$ vertex-disjoint copies of $\mathcal{A}$.
    \end{enumerate}
\end{definition}

 We remark that the property \textbf{Fac} easily holds when $\cA$ consists of a single edge, as stated in the following observation, whose proof appears in Section~\ref{sec:overview}.
\begin{observation}
\label{obs:fac}
Let $k \in \mathbb{N}$, $d \in [k-1]$ and $\cA$ be a $k$-uniform edge. 
Then, for any $\delta >0$, property $\textbf{Fac}$ holds for $\cA$ (with respect to minimum $d$-degree).
\end{observation}

\subsection{Main theorem}
We have now introduced all the necessary terminology to state our main theorem. Recall that, following Definition~\ref{def:min_degree_threshold}, the uncoloured minimum $d$-degree threshold for a Hamilton $\cA$-cycle, with $\cA$ being a link of uniformity $k$, is the smallest real number $\delta = \delta(\mathcal{A}, d)$ with the following property. For any $\alpha>0$, there exists $n_0 \in \mathbb{N}$ so that for any $n\in (m-\ell)\mathbb{N}$ with $n\geq n_0$, every $k$-uniform hypergraph $\mathcal{H}$ on $n$ vertices with $\delta_d(\mathcal{H})\geq (\delta + \alpha )n^{k-d}$ contains a Hamilton $\mathcal{A}$-cycle.

\begin{theorem}[Main theorem]
\label{thm:main}
    Let $k,\ell,m \in \mathbb{N}$ with $\ell \le m$, $\mathcal{A}$ be an $\ell$-link of order $m$ and uniformity $k$, and $d \in [k-1]$.
    Let $\delta=\delta(\mathcal{A},d)$ be the uncoloured minimum $d$-degree threshold for the containment of a Hamilton $\mathcal{A}$-cycle and
    suppose that $\mathcal{A}$ is $(\delta_0,d)$-good for some $\delta_0\geq \delta$.
    Then, for any $\alpha>0$, there exists $n_0=n_0(\cA,\alpha,d) \in \mathbb{N}$ so that for any $n\in (m-\ell)\mathbb{N}$ with $n\geq n_0$, the following holds.
    \par Let $\fH$ be a $k$-uniform hypergraph collection on vertex set $[n]$ with $|\fH|=\frac{e(\mathcal{A})-e(\mathcal{A}_s)}{m-\ell}n$ and $\delta_d(\fH)\geq (\delta_0+\alpha)n^{k-d}$. Then $\fH$ contains a transversal copy of a Hamilton $\mathcal{A}$-cycle.
\end{theorem}

Observe that the quantity $\frac{n}{m-\ell}(e(\mathcal{A})-e(\mathcal{A}_s))$ appearing in Theorem~\ref{thm:main} is precisely the number of edges in a Hamilton $\cA$-cycle covering $n$ vertices.
Therefore, it is also the size of a hypergraph collection on $[n]$ containing a transversal copy of a Hamilton $\cA$-cycle.
Moreover, if Theorem~\ref{thm:main} holds with $\delta_0=\delta$, then the family of Hamilton $\cA$-cycles is colour-blind.

\section{Tools and proof overview}
\label{sec:overview}
\subsection{Tools}
We begin with a simple proposition and a trivial observation which we are going to use often in the paper.

\begin{proposition}\label{prop:alsodirac} Let $0\leq \alpha, \delta\leq 1$, and $d, k, m,n\in \N$, with $1 \le d \le k-1$.
Let $\fH$ be a $k$-uniform hypergraph collection on vertex set $[n]$ with $|\fH|=m$ and $\delta_d(\fH)\geq \delta n^{k-d}$. Let $\cK$ be the $k$-uniform hypergraph with vertex set $[n]$, where $e$ is an edge of $\cK$ if $e\in E(H_i)$ for at least $\alpha m$ values of $i\in [m]$.
Then, $\delta_d(\cK)\geq (\delta-\alpha)n^{k-d}$.
\end{proposition}
\begin{proof}
The notation $\deg_\cK$ stands for the $d$-degree in $\cK$. 
For each $d$ pairwise distinct vertices $v_1, \dots, v_d \in [n]$, we have
\[
m\cdot \delta n^{k-d} \leq \sum_{i\in [m]}\deg_{H_i}(v_1,\dots,v_d)\leq\
m\cdot \deg_\cK(v_1,\dots,v_d)+n^{k-d} \cdot \alpha m,
\]
and therefore $\deg_\cK(v_1,\dots,v_d)\geq (\delta-\alpha)n^{k-d}$. Thus, $\delta_d(\cK)\geq (\delta-\alpha)n^{k-d}$, as wanted.
\end{proof}

\begin{observation}\label{obs:mindegreecounting}
Let $\fH$ be a $k$-uniform hypergraph collection on $V$ with $\delta_d(\fH) \ge \delta n^{k-d}$. Let $S\subseteq V$ with $|S|\leq \zeta n$. Then $\fH\setminus S$ has minimum $d$-degree at least $(\delta-\zeta) n^{k-d}$.
\end{observation}

\begin{proof}
Let $H$ be a hypergraph in $\fH$ and let $D$ be a set of size $d$ disjoint with $S$. For any vertex $v$ not in $D$, the set $D\cup\{v\}$ can have degree at most $n^{k-d-1}$. Therefore, there are at most $\zeta n^{k-d}$ edges containing $D$ which are also incident to $S$. Hence, $H\setminus S$ has minimum $d$-degree at least $(\delta-\zeta) n^{k-d}$, implying the observation.
\end{proof}

Observation~\ref{obs:mindegreecounting} allows a short proof of Observation~\ref{obs:fac}.
\begin{proof}[Proof of Observation~\ref{obs:fac}]
Let $\delta, \alpha>0$, set $\beta_0=\alpha/(2k)$, and let $\beta \le \beta_0$.
Let $\mathbf{H}$ be a hypergraph collection on $[n]$ with $\delta_d(\fH) \geq (\delta + \alpha) n^{k-d}$ and $|\mathbf{H}| \leq \beta n$. 
Suppose that we have found $s<|\fH|$ vertex-disjoint copies of $\mathcal{A}$ on $[n]$ together with a rainbow colouring (using $s$ colours), and let $S$ be the vertex set spanned by those copies.
Observe that $|S| = sk\leq \alpha n /2$.
Let $H$ be a hypergraph in $\fH$ not yet used, then by Observation~\ref{obs:mindegreecounting}, $H[V\setminus S]$ still contains an edge and thus a copy of $\cA$.
Hence we can extend the collection of copies of $\mathcal{A}$ in a rainbow fashion. This proves the observation.
\end{proof}

In the proof of Theorem~\ref{thm:main}, we need the following result that states that the vertex set of a hypergraph collections can be partitioned into linear sized sets, each preserving good minimum degree conditions in each of the hypergraphs of the collection.

\begin{lemma}
\label{lem:random}
Let $1/n \ll \alpha, \beta$ and $m\leq n^2$.
Let $t \in \mathbb{N}$ and $n_1, \dots, n_t \ge \beta n$ be integers such that $\sum_{i=1}^t n_i=n$.
Let $\fH$ be a $k$-uniform hypergraph collection on vertex set $[n]$ with $|\fH|=m$ and $\delta_d(\fH) \ge (\delta+\alpha) n^{k-d}$.
Then there exists a partition of $[n]$ into $V_1,\dots,V_t$ with $|V_i|=n_i$ for $i \in [t]$ such that any $S \in \binom{[n]}{d}$ has degree at least $(\delta + \alpha/2)n_i^{k-d}$ into $V_i$ with respect to any of the $m$ hypergraphs in $\fH$. 
\end{lemma}

We will show that a partition chosen uniformly at random has the properties required from Lemma~\ref{lem:random} with high probability.
For that, we use a concentration inequality due to McDiarmid~\cite{mcdiarmid}, whose present formulation can be found in~\cite{mcdiarmid_formulation}.

\begin{lemma}[Lemma $6.1$ in~\cite{mcdiarmid_formulation}]
\label{lem:McD}
Let $c > 0$ and let $f$ be a function defined on the set of subsets of some set $U$ such that $|f(U_1)-f(U_2)|\leq c$ whenever $|U_1|=|U_2|=m$ and $|U_1\cap U_2|=m-1$. Let $A$ be a uniformly random $m$-subset of $U$. Then for any $\alpha>0$ we have
\begin{equation*}
    \mathbb{P} \left[ |f(A)-\mathbb{E}[f(A)]|\geq \alpha c \sqrt{m} \right] \leq 2 \exp (-2\alpha^2).
\end{equation*}
\end{lemma}

Before proving Lemma~\ref{lem:random}, we prove the following consequence of the McDiarmid inequality. 

\begin{lemma}\label{lem:McDAppl}
Let $k,\ell \in \mathbb{N}$, $0<\delta' <\delta< 1$ and $1/n, 1/\ell \ll 1/k, \delta-\delta'$.
Let $H$ be a $k$-uniform $n$-vertex hypergraph with vertex set $V$ and suppose that $\deg(S,V) \geq \delta n^{k-d}$ for each $S \in \binom{V}{d}$. Let $A\subseteq V$ be a vertex set of size $\ell$ chosen uniformly at random.
Then, for every $T \in \binom{V}{d}$ we have 

\begin{equation*}
\mathbb{P}\left[\deg(T,A) < \delta' \ell^{k-d}\right] \leq 2 \exp (-\ell(\delta - \delta')^2/2).
\end{equation*}
\end{lemma}

\begin{proof}
Let $f:\mathcal{P}(V) \rightarrow \mathbb{R}$ be defined by $f(X)= \deg(T,X) $  for each $X \subseteq V$, and set $\eps = (\delta-\delta')/(2\delta) < 1/2$. Observe that $|f(U_1)-f(U_2)| \leq \ell^{k-d-1}$ for any $U_1,U_2 \in \mathcal{P}(V)$ with $|U_1|=|U_2|=\ell$ and $|U_1\cap U_2|=\ell-1$. Given an edge $e$ with $T\subseteq e$, the probability that $e\setminus T$ is contained in $A$ is at least $\frac{\binom{n-k}{\ell - k+d}}{\binom{n}{\ell}}\geq (1-\eps) \frac{\ell^{k-d}}{n^{k-d}}$, where we used $1/n, 1/\ell \ll 1/k, \delta-\delta'$. So by linearity of expectation we have $\mathbb{E}[f(A)]\geq \delta (1-\eps)\ell^{k-d}=\frac{\delta+\delta'}{2}\ell^{k-d}$. We can then apply Lemma~\ref{lem:McD} with $c=\ell^{k-d-1}$, $m=\ell$ and $\alpha=\sqrt{\ell}(\delta - \delta')/2$, and get that
$\mathbb{P}\left[f(A) < \delta'\ell^{k-d}\right] \leq 2 \exp(-\ell(\delta - \delta')^2/2)$, as desired. 
\end{proof}

We are now ready to prove Lemma~\ref{lem:random}.

\begin{proof}[Proof of Lemma~\ref{lem:random}]
  Pick a partition of $[n]$ into $V_1 \cup \cdots \cup V_t$ uniformly at random from all partitions which satisfy $|V_i| = n_i$ for all $i \in [t]$. Then by Lemma~\ref{lem:McDAppl}, we have the probability that there are $i \in [t]$, $j \in [m]$ and $S \in \binom{[n]}{d}$ such that $\deg_{H_j}(S,V_i) < (\delta + \alpha/2)n_i^{k-d}$ is at most $t \cdot m \cdot \binom{n}{d} \cdot 2 \cdot  \exp(-\alpha^2 \beta n/8)=o(1)$, where we have used that $n_i \geq \beta n$ for each $i \in [t]$, $1/n \ll \alpha, \beta$, and $m\leq n^2 $. Therefore there exists a partition with the desired properties.
\end{proof}

Finally, we state a lemma from~\cite{alp_yani_richard}, which we are going to use while proving a colour-absorption type statement.
\begin{lemma}[Lemma $3.3$ in~\cite{alp_yani_richard}]
\label{lem:absorb} Let $\alpha\in (0,1)$ and let $\ell, m, n \geq 1$ be integers satisfying $\ell\leq \alpha^7m/10^5$ and $\alpha^2n\geq 8m$. Let $K$ be a bipartite graph on vertex classes $A$ and $B$ such that $|A|=m$, $|B|=n$ and, for each $v\in A$, $\deg_K(v)\geq \alpha n$.

Then, there are disjoint subsets $B_0,B_1\subset B$ with $|B_0|=m-\ell$ and $|B_1|\geq \alpha^7n/10^5$, and the following property. Given any set $U\subset B_1$ of size $\ell$, there is a perfect matching between $A$ and $B_0\cup U$ in $K$.
\end{lemma}

\subsection{Proof overview}
As previously mentioned, the framework of the proof of our main result borrows a lot from the work of Montgomery, Müyesser, and Pehova \cite{alp_yani_richard}, and we highly recommend the reader to inspect the proof sketch given there, especially for embedding trees. We will now attempt to give a self-contained account of the main ideas of our proof strategy. For the purposes of the proof sketch, it will be conceptually (and notationally) simpler to imagine that we are trying to prove that the family of (2-uniform) Hamilton cycles is colour-blind. Observe that a Hamilton cycle is an $\cA$-cycle with $\cA$ being an edge. 
\begin{proposition}[Theorem~2 in~\cite{cheng}]\label{prop:sketcheasy}
    For any $\alpha >0$, there exists $n_0 \in \mathbb{N}$ such that the following holds.
    Let $\fG$ be a graph collection on vertex set $[n]$ with $|\fG|=n$ and $\delta(\fG) \ge (1/2 + \alpha)n$.
    Then $\fG$ contains a transversal copy of a Hamilton cycle.
\end{proposition}
\subsubsection{Colour absorption}
\par The basic premise of our approach, which is shared with \cite{alp_yani_richard}, is that Proposition~\ref{prop:sketcheasy} becomes significantly easier to prove if we assume that $|\fG|=(1+o(1))n$, that is, if we have a bit more colours than we need to find a rainbow Hamilton cycle on $n$ vertices\footnote{More intuition for why surplus colours are helpful is given in the proof sketch from \cite{alp_yani_richard}}. Thus, the starting goal of the proof is to somehow simulate having access to more colours than we need, while still starting with a graph collection of size exactly $n$. The way we achieve this is through the following lemma, which follows in a long tradition of absorption based ideas.

\begin{lemma}\label{lem:colourabsorb:general}
Let $d,k,n \in \mathbb{N}$, $1/n\ll \gamma \ll \beta \ll \alpha$ and $\delta\geq 0$. Let $\cF$ be a $k$-uniform hypergraph with $e(\cF)= \beta n$ and suppose that any $n$-vertex $k$-uniform hypergraph with minimum $d$-degree at least $\delta n^{k-d}$ contains a copy of $\cF$.
Let $\fH = (H_1, \ldots, H_m)$ be a $k$-uniform hypergraph collection on $[n]$ with $\delta_d(\fH)\geq (\delta+\alpha)n^{k-d}$ and $|\fH|=m$ with $m \geq \alpha n$.

Then, there is an uncoloured copy $\cS$ of $\cF$ in $\cup_{i\in [m]} H_i$ and disjoint sets $A,C\subset [m]$, with $|A|=e(\cF)-\gamma n$ and $|C|\geq 10\beta m$ such that the following property holds.
Given any subset $B\subset C$ with $|B|=\gamma n$, there is a rainbow colouring of $\cS$ in $\fH$ using colours in $A\cup B$.
\end{lemma}

We remark that Lemma~\ref{lem:colourabsorb:general} is the hypergraph analogue of Lemma $3.4$ from \cite{alp_yani_richard}. For the sake of completeness, we give its proof below.

\begin{proof}[Proof of Lemma~\ref{lem:colourabsorb:general}]
Let $\cH$ be the $k$-uniform hypergraph with vertex set $[n]$, where $e$ is an edge of $\cH$ exactly when $e\in E(H_i)$ for at least $\alpha m$ values of $i\in [m]$. Then, by Proposition~\ref{prop:alsodirac}, $\delta_d(\cH)\geq \delta n^{k-d}$ and, therefore, $\cH$ contains a copy of $\cF$, which we denote by $\cS$.
Observe that $\cS$ is an uncoloured copy of $\cF$ in $\cup_{i \in [m]} H_i$.

Let $K$ be the bipartite graph with vertex classes $E(\cS)$ and $[m]$, where $ei$ is an edge of $K$ exactly if $e\in H_i$. 
Note that, since each $e\in E(\cS)$ is also an edge of $\cH$, we have that $\deg_K(e)\geq \alpha m$. Then, as $\gamma\ll \beta\ll\alpha$, by Lemma~\ref{lem:absorb} with $\ell=\gamma n$, $m=\beta n$, $n=m$, there are disjoint sets $A,C\subset [m]$ with $|A|=e(\cF)-\gamma n$ and $|C|\geq 10\beta m$, such that, for any set $B\subset C$ of size $\gamma n$ there is a perfect matching between $E(\cS)$ and $A\cup B$. Note that for such a matching $M$, the function $\phi:E(\cS)\to A\cup B$, defined by $e\phi(e)\in M$ for each $e\in E(\cS)$, gives a rainbow colouring of $\cS$ in $\fH$ using colours in~$A\cup B$, as required.
\end{proof}

\subsubsection{Completing the cycle}
Lemma~\ref{lem:colourabsorb:general} provides us with a lot of flexibility, by finding a small subgraph that admits a rainbow colouring in many different ways. To prove Proposition~\ref{prop:sketcheasy}, we will also need the following proposition.
\begin{proposition}\label{prop:sketcheasier}
    Let $1/n\ll \zeta \ll \kappa, \alpha$. Let $\fG$ be a graph collection on $[n]$ with $|\fG|=(1+\kappa-\zeta)n$ and $\delta(\fG)\geq (1/2+\alpha)n$. Let $a,b\in [n]$ be distinct vertices. Then, $\fG$ contains a rainbow Hamilton path with $a$ and $b$ as its endpoints, using every colour $G_i$ with $i\in [(1-\zeta)n]$.
\end{proposition}
Proposition~\ref{prop:sketcheasier}, in combination with Lemma~\ref{lem:colourabsorb:general}, gives a proof of Proposition~\ref{prop:sketcheasy}. 
\begin{proof}[Sketch of Proposition~\ref{prop:sketcheasy}] Let $\mathcal{C}$ denote the set of the $n$ colours. Apply Lemma~\ref{lem:colourabsorb:general} with $\mathcal{F}$ being a path of length $\beta n$ (and some constant $\gamma\ll \beta \ll \alpha$).
This gives a path $\mathcal{S}$ in $\fG$ and colour sets $A$ and $C$ (and we can fix $C$ to be a subset of the original set of size $\rho n$ with $\gamma \ll \rho \ll \beta$). Let $a$ and $b$ be the endpoints of $\mathcal{S}$. Set $\fG'$ to be the graph collection obtained by restricting $\fG$ to the vertex set $([n] \setminus V(\mathcal{S})) \cup \{a,b\}$ and colour set $\mathcal{C}\setminus A$.
Apply Proposition~\ref{prop:sketcheasier}, labelling the colours in $\fG'$ so that the first $(1-\zeta)n$ colours correspond to those in $\mathcal{C}\setminus (A \cup C)$. This way, we extend $\mathcal{S}$ to a Hamilton cycle $\mathcal{H}$. While the edges in $\cS$ are still uncoloured, those in $\mathcal{H}\setminus \mathcal{S}$ have been assigned a colour set using all colours in $\mathcal{C}\setminus (A \cup C)$ and exactly $|C|-\gamma n$ colours from $C$. Using the absorption property of $\mathcal{S}$, the path $\cS$ can be given a colouring using all the colours in $A$ and the remainder colours in $\mathcal{C}$, thereby giving $\mathcal{H}$ a rainbow colouring, as desired.
\end{proof}
\par Unfortunately, due to the technicalities present in the statement, Proposition~\ref{prop:sketcheasier} is far from trivial to show. Most of the novelty in the proof of our main theorem is the way we approach Proposition~\ref{prop:sketcheasier} for arbitrary $\mathcal{A}$-chains satisfying $\textbf{Ab}$, $\textbf{Con}$, and $\textbf{Fac}$. We now proceed to explain briefly how we achieve this, and how the three properties come in handy. 
\par Firstly, in the setting of Proposition~\ref{prop:sketcheasier}, it is quite easy to find a few rainbow paths using most of the colours from the set $[(1-\zeta)n]$. Below is a formal statement of a version of this for arbitrary $\mathcal{A}$-chains, where we remark that $\left(\frac{e(\mathcal{A})-e(\mathcal{A}_s)}{m-\ell} \right)n$ is the number of edges of an $\cA$-cycle on $n$ vertices.
\begin{lemma}
\label{lem:rainbow_paths_tiling}
Let $1/n \ll 1/T \ll \omega, \alpha, 1/m$. Let $\mathcal{A}$ be an $\ell$-link of order $m$ and uniformity $k$, and $d\in [k-1]$. Let $\delta$ be the minimum $d$-degree threshold for the containment of a Hamilton $\mathcal{A}$-cycle. 
Let $\fH$ be a $k$-uniform hypergraph collection on $[n]$ with $\delta_d(\fH)\geq (\delta+\alpha)n^{k-d}$, and suppose that $|\fH|\geq \left(\frac{e(\mathcal{A})-e(\mathcal{A}_s)}{m-\ell} \right)n$.
Then $\fH$ contains a rainbow collection of $T$-many pairwise vertex-disjoint $\mathcal{A}$-chains covering all but at most $\omega n$ vertices of $\fH$. 
\end{lemma}

\begin{proof}
Choose $\omega, T$ such that Lemma~\ref{lem:random} holds with $\beta=(1-\omega/2)/T$, and set $t=\left(\frac{e(\mathcal{A})-e(\mathcal{A}_s)}{m-\ell} \right)$.
Let $\fH$ be a $k$-uniform hypergraph collection on $[n]$ with $\delta_d(\cH) \ge (\delta+\alpha)n^{k-d}$ and $|\fH| \ge tn$.

By Lemma~\ref{lem:random} applied with $\beta=(1-\omega/2)/T$, there exists a partition of $[n]$ into $V_1, \dots, V_T, V_{T+1}$ with $|V_1|=\dots=|V_T|=(1-\omega/2)n/T$ and $|V_{T+1}|=\omega n/2$, such that for any $1 \le i \le T+1$ and any hypergraph $\cH$ of the collection $\fH$, it holds that $\delta_d(\cH[V_i]) \ge (\delta+\alpha/2) |V_i|^{k-d}$.
We claim that we can greedily cover all but at most $m \cdot T$ vertices of $V_1, \dots, V_T$ with a rainbow collection of $T$-many pairwise vertex-disjoint $\cA$-chains $\cA_1, \dots, \cA_T$, such that $\cA_i$ covers all but at most $m$ vertices of the set $V_i$ for each $i \in [T]$.
Suppose we were able to do so for the sets $V_1, \dots, V_i$ for some $1 \le i < T$.
Then the number of colours used so far is at most $i \cdot (t|V_1|)$ and thus there are at least $tn-(T-1)t|V_1| = tn\frac{(T-1)\omega+2}{2T}$ available colours.
Let $C$ be the set of such colours.
Observe that a rainbow $\cA$-chain covering the vertices of $V_{i+1}$ uses no more than $t (1-\omega) n/T = \eta |C|$ colours, where $\eta=\frac{2-\omega}{(T-1)\omega+2} \le \frac{\alpha}{4}$, where we used $1/T \ll \omega, \alpha$ for the last inequality.
Let $\cK$ be the $k$-uniform hypergraph with vertex set $V_{i+1}$, where $e$ is an edge of $\cK$ if $e \in E(H_i)$ for at least $\eta |C|$ colours $i \in C$.
Then by Proposition~\ref{prop:alsodirac}, we have $\delta_d(\cK) \ge (\delta+\alpha/2-\eta) |V_{i+1}|^{k-d} \ge (\delta + \alpha/4) |V_{i+1}|^{k-d}$, where we used $\eta \le \alpha/4$ for the last inequality.
Therefore $\cK$ contains a copy of a Hamilton $\cA$-cycle, which in turn contains an $\cA$-chain covering all but $m$ vertices of $V_{i+1}$.
Now we greedily assign colours from $C$ to this $\cA$-chain in a rainbow fashion.

This shows we can find a rainbow collection of $T$-many pairwise vertex-disjoint $\cA$-chains covering all but at most $m \cdot T + |V_{T+1}| \le \omega n$ vertices of $\fH$, as wanted.
\end{proof}

Although it is easy to use most of the colours coming from a colour set using the above result, a challenge in Proposition~\ref{prop:sketcheasier} is that we need to use all of the colours coming from the set $[(1-\zeta)n]$. As we are currently concerned with the case when $\mathcal{A}$ consists of a single edge, this will not be a major issue. Indeed, using the minimum degree condition on each of the colours, we can greedily find rainbow matchings using small colour subsets of $[(1-\zeta)n]$ (see Observation~\ref{obs:fac}). For arbitrary $\mathcal{A}$, we would like to proceed in the same way; however, say when $\mathcal{A}$ is a triangle, the situation becomes considerably more complicated. This is why the property \textbf{Fac} is built into the assumptions of the main theorem.  

\par Our ultimate goal is to build a single $\mathcal{A}$-chain connecting specific ends, not just a collection of $\mathcal{A}$-chains. Hence, we rely on the property \textbf{Con} to connect the ends of the paths we obtained via Lemma~\ref{lem:rainbow_paths_tiling} (as well as the greedy matching we found for the purpose of exhausting a specific colour set). An issue is that \textbf{Con} is an uncoloured property, whereas we would like to connect these ends in a rainbow manner. Here we rely on the trick offered by Proposition~\ref{prop:alsodirac}, which states that in hypergraph collections where each hypergraph has good minimum $d$-degree conditions, we can pass down to an auxiliary hypergraph $\mathcal{K}$ which also has good minimum $d$-degree conditions. An edge appears in $\mathcal{K}$ if and only if that edge has $\Omega(n)$ many colours in the original hypergraph collection. We can use the property \textbf{Con} on $\mathcal{K}$ to connect ends via short uncoloured paths, and later assign greedily one of the many available colours to the edges on this path.
\par As is the case with many absorption-based arguments, the short connecting paths we find will be contained in a pre-selected random set. After all the connections are made, there will remain many unused vertices inside this random set. To include these vertices inside a path, we use the property \textbf{Ab}. Similarly to \textbf{Con}, property \textbf{Ab} is an uncoloured property, but we can use again the trick of passing down to an appropriately chosen auxiliary graph $\mathcal{K}$.

\section{Proof of main theorem}
\label{sec:main_proof}

\begin{proof}[Proof of Theorem~\ref{thm:main}]
    Let $k,\ell,m \in \mathbb{N}$ with $\ell \le m$, $\mathcal{A}$ be an $\ell$-link of order $m$ and uniformity $k$, and $d \in [k-1]$.
    Let $\delta:=\delta(\mathcal{A},d)$ be the minimum $d$-degree threshold for the containment of a Hamilton $\mathcal{A}$-cycle, and suppose that $\mathcal{A}$ is $(\delta_0,d)$-good for some $\delta_0\geq \delta$.
    In the following, the constant implicit in any $O(\cdot)$ only depends on $\mathcal{A}$ and, therefore, can be bounded in terms of $m$.
    
    \textbf{Constants.}
    Let $\alpha>0$, let $c$ be given by $\mathbf{Con}$ with $\alpha/10$ and let $\beta_0$ be given by $\mathbf{Fac}$ applied with $\alpha/6$.
    Choose $\beta < \beta_0$ such that $0<\beta\ll \alpha,1/c,1/m$.
    Next choose $\rho$ and $\gamma$ such that $0<\gamma \ll \rho \ll \beta$ and the hierarchy in Lemma~\ref{lem:colourabsorb:general} is satisfied with $\gamma,\beta,\alpha$.
    Let $\tau$ and $\eta$ be given by $\mathbf{Ab}$ with $\alpha=\rho$, so that we have $0<\tau,\eta\leq \rho$.
    Now choose $T\in \mathbb{N}$ and $\omega, \nu>0$ with $ 1/T \ll \omega \ll \nu \ll \eta$ so that the hierarchy in Lemma~\ref{lem:rainbow_paths_tiling} is satisfied with $T,\omega,\alpha$.
    Finally, let $n \in (m-\ell) \mathbb{N}$ be such that $1/n \ll 1/T$ and $1/n \ll 1/n_0$ for any of the $n_0$ coming from the applications of $\mathbf{Con}$, $\mathbf{Fac}$ and $\mathbf{Ab}$ above.
    Without loss of generality we assume that $\beta n$ is an integer and that there exists an $\mathcal{A}$-chain on $\beta n$ edges.
    We summarise the dependency between the parameters as follows
    \[1/n \ll 1/T\ll \omega\ll  \nu \ll \eta , \tau, \gamma \ll \rho\ll \beta\ll \alpha,1/c,1/m \, . \]
    
    \textbf{Set-up.} 
    Let $\fH$ be a $k$-uniform hypergraph collection on vertex set $[n]$ with $|\fH|=\frac{e(\mathcal{A})-e(\mathcal{A}_s)}{m-\ell}n$ and $\delta_d(\fH)\geq (\delta+\alpha)n^{k-d}$.
    Set $t:=\frac{e(\mathcal{A})-e(\mathcal{A}_s)}{m-\ell}$ so that $tn=|\fH|$. We will use $[tn]$ to refer to our set of colours. Set $V:=V(\fH)$.
    
    For an easier navigation of the proof, the reader can refer to Figure~\ref{fig:proof}.
    \begin{figure}[h]
        \begin{center}
        \includegraphics[scale=0.49]{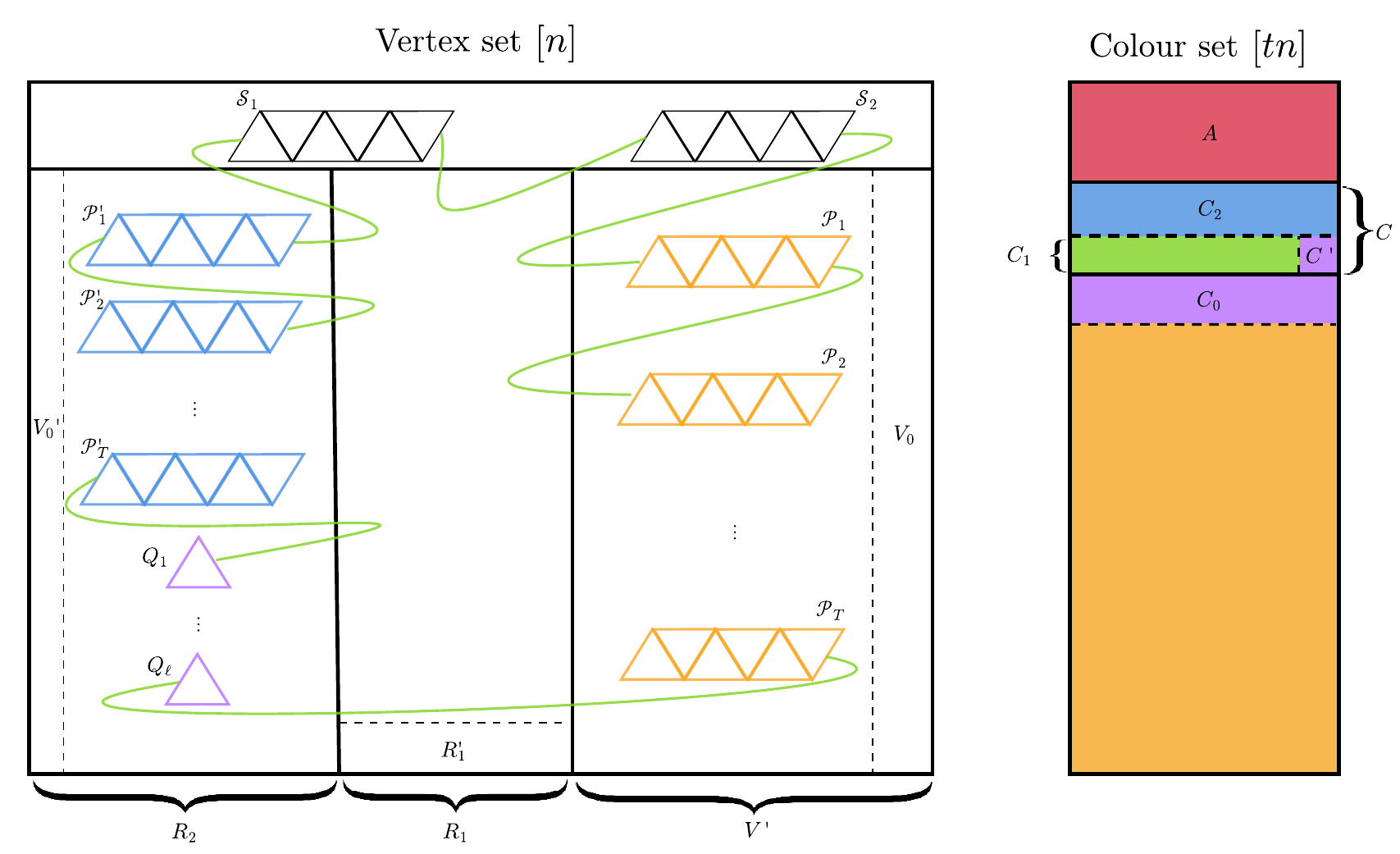}
	    \end{center}
	    \captionsetup{font=footnotesize}
	    \captionof{figure}{
	    The picture illustrates the partitions of the vertex set $[n]$ and the colour set $[tn]$ used in the proof of Theorem~\ref{thm:main} while building a rainbow Hamilton $\cA$-cycle.
	    Each triangle stands for a copy of the building block $\cA$ and each square of a path stands for a copy of an $\cA$-chain.
	    The green curved lines denote the connections between $\cA$-chains and are $\cA$-chains, and we remark that their internal vertices all belong to $R_1$.
	    Additionally, the colour of an edge refers to the subset of colours from which the colour assigned to the edge is taken; for example, the blue edges are those which get a colour from the subset $C_2$. 
	    Moreover, we left $\cS_1$ and $\cS_2$ uncoloured to stress that their edge colours are only assigned at the end of the proof. \\ \hspace{\textwidth}
        $\cS_1$ is the colour absorber for the colour sets $A$ and $C=C_1 \cup C_2$ (Step $1$, see \eqref{colour_absorber_S1});
        $\cS_2$ is the vertex absorber with colours in $C_1$ (Step $2$, see \eqref{vertex_absorber_S2});
	    $R_1$ is the reservoir connector via colours in $C_1$ (Step $3$, see \eqref{connector_R1});
        $R_2$ is a vertex set taken to balance colours (Step $4$). \\ \hspace{\textwidth}
        The family $\{\cP_i:i \in [T]\}$ of rainbow $\cA$-chains almost covers vertices in $V'$ and colours in $[tn] \setminus (A \cup C)$, with $V_0$ being the set of unused vertices and $C_0$ being the set of unused colours (Step $5$).
        The set $C' \subseteq C_1$ is a minimal size subset of $C_1$ to make $|C' \cup C_0|$ divisible by $e(\cA)$ and the rainbow collection $\{\mathcal{Q}_i:i \in [\ell]\}$ of copies of $\cA$ exhausts $C' \cup C_0$ inside $R_2$ (Step $6$).
        Moreover the family $\{\cP_i':i \in [T]\}$ of rainbow $\cA$-chains almost covers the leftover vertices in $R_2$ using colours in $C_2$, with $V_0'$ being the set of unused vertices (Step $7$).\\ \hspace{\textwidth}
        We connect into an almost spanning $\cA$-cycle all the $\cA$-chains and copies of $\cA$ built so far, using vertices in $R_1$ and colours in $C_1$ (Step $8$).
        We then absorb the leftover vertices via $\cS_2$ into a longer $\cA$-chain (with the same ends) and assign to its edges colours from $C_1$ (Step $9$).
        Finally we assign the colours of $A$ and the leftover colours of $C$ to $\cS_1$ (Step $10$).\\ \hspace{\textwidth}
        This gives a rainbow Hamilton $\cA$-cycle.}
	    \label{fig:proof}
    \end{figure}
    
    \textbf{1. Setting up the colour absorber.} Let $\cF$ be an $\cA$-chain on $\beta n$ edges (and thus $\frac{\beta n - e(\cA_s)}{e(\cA)-e(\cA_s)} (m - \ell) + \ell=\beta n/t+O(1)$ vertices), which exists by our choice of $\beta$. As the minimum degree threshold for the containment of $\cF$ is at most $\delta$, the hypotheses of Lemma~\ref{lem:colourabsorb:general} are satisfied for $\cF$ with the hypergraph collection $\fH$ (observe also that we may assume without loss of generality that $|\fH|=tn\geq \alpha n$). Therefore, there exist disjoint colour sets $A, C \subseteq [tn]$ with $|A|=e(\cF)-\gamma n=(\beta-\gamma)n$ and $|C|\geq 10\beta tn$, and an uncoloured copy $\cS_1$ of $\cF$ in $\fH$ such that the following holds.
    \begin{equation}
        \label{colour_absorber_S1}
        \text{\parbox{.92\textwidth}{
        Given any subset $B\subseteq C$ with $|B|=\gamma n$, there is a rainbow colouring of $\cS_1$ in $\fH$ using colours in $A\cup B$.
        }}
    \end{equation}
    We denote the start and the end of $\cS_1$ by $\cF_1$ and $\cF_2$, respectively, and define $S_1':=V(\cS_1) \setminus (V(\cF_1) \cup V(\cF_2))$ to be the set of all vertices of $\cS_1$ except those in its ends.
    Note that, as $\rho \ll \beta$, we have $10\beta t n \geq \rho n$ and, without relabelling, we can fix $C$ to be a subset of the original set $C$ of size exactly $\rho n$. For convenience, we split the set $C$ arbitrarily into two subsets $C_1$ and $C_2$, with $|C_1|=\gamma n/2$ (and $|C_2|=(\rho-\gamma/2)n$). Our goal in the remainder of the proof, in correspondence with Proposition~\ref{prop:sketcheasier} from the proof overview, is to find a rainbow $\mathcal{A}$-chain, vertex-disjoint with $S_1'$, starting in $\mathcal{F}_2$ and ending in $\mathcal{F}_1$, using all colours in $[tn]\setminus (C\cup A)$, and some colours from $C$. Note that, similarly to the setting of Proposition~\ref{prop:sketcheasier}, we have $(t - (\beta - \gamma))n$ colours available compared to $(t-\beta)n$ edges that we need to colour.
    
    \textbf{2. Setting up the vertex absorber.} 
    Since $|V(\mathcal{S}_1)|= \beta n/t + O(1) \leq \alpha n/1000$, where we used that $\beta\ll \alpha$ in the last inequality, we have that $\delta_d(\fH[V\setminus V(\mathcal{S}_1)])\geq (\delta+9\alpha/10)n^{k-d}$ by Observation~\ref{obs:mindegreecounting}. We define an auxiliary graph $\cK_1$ to be the $k$-uniform graph on vertex set $V_1 := V \setminus V(\cS_1)$, where $e$ is an edge of $\cK_1$ if and only if $e \in E(\cH_i)$ for at least $\alpha |C_1|/2=\alpha\gamma n/4$ values of $i \in C_1$. Then, using Proposition~\ref{prop:alsodirac} on $\cK_1$, we get that $\delta_d(\cK_1) \geq (\delta+\alpha/2)n^{k-d}\geq (\delta+\rho)n^{k-d}$. By the choice of the constants $\eta$ and $\tau$ for \textbf{Ab}, we have that there exists a set $S_2 \subseteq V_1$ of size at most $\tau n$ such that the following property holds. 
    \begin{equation}
        \label{vertex_absorber_S2}
        \text{\parbox{.92\textwidth}{
        For any set $L \subseteq V_1 \setminus S_2$ of size at most $\eta n$ with $|L|\in (m-\ell)\mathbb{N}$,
        there exists an embedding of an $\cA$-chain in $\cK_1$ with vertex set $S_2\cup L$.
        Furthermore, the embedding of the start and the end of this $\cA$-chain does not depend on the subset $L$.}}
    \end{equation}
    In particular, by taking $L=\emptyset$ in~\eqref{vertex_absorber_S2}, there is a copy $\cS_2$ of an $\cA$-chain in $\cK_1$ with vertex set $S_2$. 
    We denote its ends by $\cG_1$ and $\cG_2$, and define $S_2':=V(\cS_2) \setminus (V(\cG_1) \cup V(\cG_2))$. 
    
    \textbf{3. Setting up the reservoir connector.} By Observation~\ref{obs:mindegreecounting}, we have that $\delta_d(\fH[V\setminus (S_1' \cup S_2')])\geq (\delta+\alpha/2)n^{k-d}$, where we used that $\tau, \beta \ll\alpha$. We define another auxiliary graph $\cK_2$ as the $k$-uniform graph on vertex set $V_2 := V\setminus (S_1' \cup S_2')$, where $e$ is an edge of $\cK_2$ if an only if $e \in E(\cH_i)$ for at least $\alpha|C_1|/2=\alpha\gamma n/4$ values of $i \in C_1$. By Proposition~\ref{prop:alsodirac}, we know that $\delta_d(\cK_2)\geq (\delta+\alpha/3)n^{k-d}$.
    Using Lemma~\ref{lem:random} on $\cK_2$ with $t=2$, $n_1=\nu n$ and $n_2=|V_2| - n_1$, we get a set $R_1$ of size $\nu n$, such that every subset of $d$ vertices of $V_2$ have $d$-degree at least $(\delta+\alpha/6)n_1^{k-d}$ into $R_1$ in the graph $\cK_2$. 
    Moreover, we can assume that $R_1$ does not contain any of the vertices in $V(\cF_1) \cup V(\cF_2) \cup V(\cG_1) \cup V(\cG_2)$.
    From Observation~\ref{obs:mindegreecounting}, we have that for any two vertex-disjoint copies $\mathcal{S}$ and $\mathcal{T}$ of $\cA_s$ in $V_2$ and any $R' \subseteq R_1$ of size $|R'|\leq \alpha n_1/50$, we have that the minimum $d$-degree in $(R_1 \setminus R') \cup V(\mathcal{S}) \cup V(\mathcal{T})$ in $\cK_2$ is at least $(\delta+\alpha/10)n_1^{k-d}$.
    Then, property \textbf{Con} applied to the hypergraph $\cK_2[(R_1 \setminus R') \cup V(\mathcal{S}) \cup V(\mathcal{T})]$ implies the following.
   \begin{equation}
        \label{connector_R1}
        \text{\parbox{.92\textwidth}{ 
        For any $R' \subseteq R_1$ of size $|R'| \le \alpha n_1/50$ and any two vertex-disjoint copies $\mathcal{S}$ and $\mathcal{T}$ of $\cA_s$ in $\mathcal{K}_2[V_2 \setminus R']$, there is an $\cA$-chain of length at most $c$ in $(R_1 \setminus R') \cup V(\mathcal{S}) \cup V(\mathcal{T})$ in $\cK_2$ with start $\mathcal{S}$ and end $\mathcal{T}$.
        }}
    \end{equation}

 \par \textbf{4. Setting aside a random set to balance vertices and colours.} Define $n_0:=n-\frac{|A|+|C|}{t}$ and $r_2$ so that the equality below holds\footnote{Ignoring divisibility issues, $n_0$ represents the number of vertices an $\mathcal{A}$-cycle on $tn-|A|-|C|$ edges would have.}
 \begin{equation*}
   n-|V(\mathcal{S}_1)|-|S_2|-|R_1|-r_2 = n_0 \, .
 \end{equation*}
 In particular, we have that 
 $$r_2=\frac{|A|+|C|}{t}-\frac{\beta n}{t} - |S_2|-\nu n +O(1)=\frac{\rho-\gamma}{t}n - |S_2|-\nu n +O(1).$$
 As $0\leq |S_2|\leq \tau n$, $\tau,\nu \ll \rho$, and $\gamma \ll \rho$ we have that 
 \begin{equation}\label{boundr2} \frac{\rho}{2t}n \leq r_2\leq \frac{\rho-\gamma}{t}n +O(1).\end{equation}

 As $\nu,\tau,\beta \ll \alpha$, by Observation~\ref{obs:mindegreecounting}, we have that $\delta(\fH[V \setminus (V(\cS_1) \cup V(\cS_2) \cup R_1)])\geq (\delta+\alpha/3)n^{k-d}$. 
 Using Lemma~\ref{lem:random} on $\fH[V \setminus (V(\cS_1) \cup V(\cS_2) \cup R_1)]$ with $t=2$, we find a subset $R_2$ of $V \setminus (V(\cS_1) \cup V(\cS_2) \cup R_1)$ of size $r_2$, so that every subset of $d$ vertices of $V \setminus (V(\cS_1) \cup V(\cS_2) \cup R_1)$ have $d$-degree $(\delta+\alpha/6) r_2^{k-d}$ into $R_2$ with respect to each of the hypergraphs in the collection.
 
\textbf{5. Covering most of the leftover vertices via $\mathcal{A}$-chains using almost all the colours in $[tn]\setminus (A\cup C)$.} 
    Set $V':=V \setminus (V(\cS_1) \cup V(\cS_2) \cup R_1\cup R_2)$ and note $|V'|=n_0$. Let $\fH'$ be the hypergraph collection obtained by restricting $\fH$ to the vertex set $V'$ and colour set $[tn]\setminus(A\cup C)$. Using the upper bound from \eqref{boundr2} and that $\nu, \tau, \rho, \beta \ll \alpha$, we have through Observation~\ref{obs:mindegreecounting} that $\delta(\fH')\geq (\delta + \alpha/8)n_0^{k-d}$. Moreover, by our definition of $n_0$, we have $|\fH'| = tn - |A| - |C| = t n_0 = \left(\frac{e(\mathcal{A})-e(\mathcal{A}_s)}{m-\ell} \right)n_0$. Therefore $\fH'$ satisfies the hypotheses of Lemma~\ref{lem:rainbow_paths_tiling} and we find a rainbow collection $\{\mathcal{P}_i:i \in [T]\}$ of $T$-many vertex-disjoint $\mathcal{A}$-chains in $\fH'$, covering all but a vertex subset $V_0$ of size at most $\omega n_0$, and using only colours from $[tn] \setminus (A \cup C)$. 
    Moreover, observe that the set of colours from $[tn]\setminus(A\cup C)$ unused by $\bigcup_{i \in [T]} \mathcal{P}_i$, which we denote by $C_0$, has size at most $t\omega n_0+O(T) \le 2t \omega n_0$.
    
\textbf{6. Exhaust $C_{0}$ inside $R_2$.}
    Let $C' \subseteq C_1$ be a minimal size subset of $C_1$ such that $|C'\cup C_0|$ is divisible by $e(\mathcal{A})$. Note this can be accomplished with a subset $C'$ satisfying $|C'|=O(1)$ and, since $|C_0| \le 2 t\omega n_0$, we have that $|C_0\cup C'|\leq 2t\omega n_0 + O(1)$. Let $\fH''$ be the hypergraph collection obtained by restricting $\fH$ to the vertex set $R_2$ and colour set $C_0 \cup C'$. Recall that, by property of the set $R_2$, we have that $\delta(\fH'')\geq (\delta+\alpha/6){r_2}^{k-d}$. As $1/n \ll \rho$ and $n \gg n_0$, we have that $r_2$ is sufficiently large to apply \textbf{Fac} and deduce that $\fH''$ contains a rainbow collection $\{\mathcal{Q}_i:i \in [\ell]\}$ of $\ell=|C_0 \cup C'|/e(\cA)$ vertex-disjoint copies of $\cA$, using all of the colours in $C_{0}\cup C'$ and $|C_0 \cup C'| \le 3t\omega n \le (6t^2\omega/\rho)r_2 \le (\alpha/1000)r_2$ vertices of $R_2$, where we used the lower bound on $r_2$ in \eqref{boundr2} and $\omega \ll \rho$.

\textbf{7. Shrink leftover vertices in $R_2$ via $C_2$.} Let $R_2'$ be the subset of $R_2$ consisting of those vertices unused in the previous step and set $r_2':=|R_2'|$.
Since $r_2-r_2'\leq (\alpha/1000)r_2$ and using Observation~\ref{obs:mindegreecounting}, we have that $\delta(\fH''[R_2'])\geq (\delta+\alpha/30){r_2'}^{k-d}$. Note that by the upper bound in \eqref{boundr2} we have that $$tr_2'\leq tr_2\leq (\rho-\gamma)n +O(1)\leq (\rho-\gamma/2)n=|C_2|.$$
Let $\fH'''$ be the hypergraph collection obtained by restricting $\fH$ to the vertex set $R_2'$ and colour set $C_2$. Then $|\fH'''| = |C_2| \geq \left(\frac{e(\mathcal{A})-e(\mathcal{A}_s)}{m-\ell} \right) r_2'$. Hence, similarly to Step~5, we can apply
Lemma~\ref{lem:rainbow_paths_tiling} to $\fH'''$ in order to find a rainbow collection $\{\mathcal{P}'_i:i \in [T]\}$ of $T$-many vertex-disjoint $\mathcal{A}$-chains (with colours coming from $C_2$) in $\fH'''$, covering all but a vertex subset $V_0'\subseteq R_2'$ of size at most $\omega  r_2'$.

\textbf{8. Connect everything via $C_1$ and $R_1$ to build an almost spanning $\cA$-cycle.}
    We recall that we built one uncoloured $\cA$-chain in each of Step~1 and~2, $|C_0 \cup C'|/e(\cA)$ rainbow $\cA$-chains in Step~6 (indeed a copy of $\cA$ is trivially a rainbow $\cA$-chain of length $1$), and $T$ rainbow $\cA$-chains in each of Step~5 and~7.
    Therefore, at this point there are $2+2T+|C_0 \cup C'|/e(\cA) \le 3t \omega n$ vertex-disjoint $\cA$-chains, which we will now connect to build an $\cA$-cycle, using additional vertices in $R_1$ and colours in $C_1$. This can be done by repeatedly invoking property~\eqref{connector_R1}. Indeed, suppose that the chains are labelled $\mathcal{J}_1,\cdots, \mathcal{J}_z$ where $z\leq 3t \omega n$. Suppose that for some $1\leq z'\leq z$, we found an $\cA$-chain $\mathcal{R}$ in $\mathcal{K}_2$ such that the following properties all hold.
    \begin{itemize}
        \item $V(\mathcal{R})\supseteq \bigcup_{i\in[z']} V(\mathcal{J}_i)$;
        \item With $R':= V(\mathcal{R})\setminus \bigcup_{i\in[z']} V(\mathcal{J}_i)$, we have $R' \subseteq R_1$ and $|R'| \le ((m-\ell)c+\ell)z'$;
        \item The start of $\mathcal{R}$ is the start of $\mathcal{J}_1$, and its end is the end of $\mathcal{J}_{z'}$.
    \end{itemize}
    \par We remark that for $z'=1$, the $\cA$-chain $\mathcal{R}=\mathcal{J}_1$ has the above properties.
    As $|R'|\le \alpha n_1/50=\alpha \nu n/50$ (using that $\omega\ll \alpha,1/c,1/m$), property~\eqref{connector_R1} applies to show that there is an $\cA$-chain in $\mathcal{K}_2$ of length at most $c$ starting in the end of $\mathcal{J}_{z'}$ and ending in the start of $\mathcal{J}_{z'+1}$ (where $z'+1=1$ if $z'=z$), which uses vertices of $R_1 \setminus R'$ and these shared ends. 
    Observe that these chain uses fewer than $(m-\ell)c+\ell$ vertices of $R_1 \setminus R'$. 
    This shows that $\mathcal{R}$ can be extended to satisfy the above properties with respect to $z'+1$.
    Inductively, we obtain an $\mathcal{A}$-cycle $\mathcal{C}_0$ in $\mathcal{K}_2$ covering $\bigcup_{i\in[z]} \mathcal{J}_i$, and we denote by $R_1'$ the vertices from $R_1$ unused by $\mathcal{C}_0$.

    \par Consider the set of edges of $\mathcal{C}_0$ not contained in some $\mathcal{J}_i$, i.e. the edges we have found in the previous steps to connect the various $\mathcal{J}_i$'s. Note that there are at most $((m-\ell)c+\ell)zt \le \alpha\gamma n/1000$ such edges, where we used $\omega \ll \gamma, \alpha, 1/c, 1/m$. Moreover, each such edge belongs to $\cK_2$ and thus has at least $\alpha|C_1|/2=\alpha\gamma n/4$ colours coming from $C_1$. Therefore we can greedily assign a distinct colour of $C_1$ to each such edge.
    
    \item \textbf{9. Absorb the leftover vertices.} Note that $\mathcal{C}_0$ covers everything in $V$, except the sets $V_0$, $V_0'$, and $R_1'$ which are leftover from Steps~5,~7, and~8, respectively. Note that $|R_1'| + |V_0| + |V_0'| \le \nu n + \omega n_0 + \omega r_2' \le (\nu+2\omega)n \le \eta n$, where we used that $\omega, \nu \ll \eta$.
    Therefore, by~\eqref{vertex_absorber_S2}, there exists an embedding of an $\cA$-chain $\mathcal{S}_2'$ in $\cK_1$ with vertex set $S_2 \cup R_1' \cup V_0 \cup V_0'$, and with the same ends as the $\cA$-chain $\cS_2$. As in the previous step, we can then colour the edges of $\mathcal{S}_2'$ in a rainbow fashion, by assigning colours still available in $C_1$. This is possible as $|V(\cS_2)| \le (\tau + \eta)n$ and thus there are at most $t(\tau+\eta)n\le \alpha \gamma n/10$ new edges, where we used $\tau, \eta \ll \gamma, \alpha$. 
    Moreover these edges belong to the hypergraph $\cK_1$ and appear in at least $\alpha \gamma n/4$ colours in $C_1$, while we only used at most $\alpha\gamma n/1000$ colours from $C_1$ in the previous step.
    Therefore there are at least $\alpha \gamma n/8$ available colours for each edge, and we can greedily assign distinct colours.  
    
    \textbf{10. Assign a colouring to $\cS_1$.}
    Observe that we now have a Hamilton $\cA$-cycle that is rainbow except for $\cS_1$, which is still uncoloured.
    Moreover, we have used all colours outside $A \cup C$ and some colours in $C$, and we have not used any of the colours in $A$.
    Therefore the unused colours must be those in $A$ together with a subset $B \subseteq C$ of size $\gamma n$.
    We can then assign colours to $\cS_1$ in a rainbow fashion by property~\eqref{colour_absorber_S1}.
    This completes the rainbow embedding and finishes the proof.
\end{proof}

\section{Applications of the main theorem and proof of Theorem~\ref{thm:applications}}
\label{sec:applications}

In this section we discuss some applications of our main theorem and, in particular, we prove Theorem~\ref{thm:applications}.
The proofs of the statements of Theorem~\ref{thm:applications} all follow the same strategy. Suppose we want to prove $d$-colour-blindness of a family $\cF$. 
We first identify a link $\cA$ such that each member of $\cF$ is an $\cA$-cycle.
Then we show that $\cA$ is $(\delta,d)$-good, with $\delta$ being the uncoloured minimum degree threshold for the family $\cF$.
Once this is done, the $d$-colour-blindness of $\cF$ is a consequence of Theorem~\ref{thm:main}.
We will give a full proof of the statement~\ref{thm:app_HC2} of Theorem~\ref{thm:applications} with $r=2$, while we will only sketch how to prove properties $\ab$, $\con$, and $\fac$ for the statement~\ref{thm:app_HC2} with $r>2$ and the statement~\ref{thm:app_cycle}.
The reader can then easily complete a full proof, by mimicking the one given for the square of Hamilton cycles. 

\subsection{Powers of Hamilton cycles}

The (uncoloured) minimum degree threshold for the containment of the $r$-th power of a Hamilton cycle in a $2$-uniform graph was conjectured to be $\tfrac{r}{r+1} n$ by Pósa (for $r=2$) and Seymour (for larger $r$).
This was proved by Komlós, Sárkozy, and Szemerédi~\cite{square_Ham_cycle,powers_Ham_cycle}, using the regularity method and the Blow-Up Lemma.
Later, Levitt, Sárkozy, and Szemerédi~\cite{square_ham_noreg} obtained a proof for the case $r=2$ that avoids the regularity lemma and is instead based on the absorption method.
More recently, Pavez-Signé, Sanhueza-Matamala, and Stein~\cite{pavez2021towards} generalised this to $r \ge 2$, while studying the hypergraph version of the problem.
Both of these fit our framework and allow us to obtain part~\ref{thm:app_HC2} of Theorem~\ref{thm:applications}.
We will first focus on the case $r=2$, which we will use as a more detailed example and can be reformulated as follows.

\begin{theorem}[Rainbow version of Pósa's conjecture]
\label{thm:posa_rb}
For any $\alpha>0$ there exists $n_0$ such that for $n \ge n_0$ the following holds.
Any graph collection $\fG$ on vertex set $[n]$ with $\delta(\fG) \ge (2/3 + \alpha)n$ contains a transversal copy of the square of a Hamilton cycle.
\end{theorem}


\par As mentioned above, in order to prove Theorem~\ref{thm:posa_rb}, it is enough to show that the square of a cycle is an $\cA$-cycle for a suitable choice of a $(2/3,1)$-good link $\cA$. 
Towards that goal, we let $\cA$ be the $2$-link coming from an arbitrary ordering of $K_3$ (see Figure~\ref{fig:square}) and we prove that such $\cA$ is indeed $(2/3,1)$-good. 
The properties {\bf Ab} and {\bf Con} for $\cA$ follow from the proof of the (uncoloured) Pósa conjecture in~\cite{square_ham_noreg}. 
In that proof, the authors give an exact version of the uncoloured threshold, by distinguishing an extremal and a non-extremal case. 
They say that a graph is extremal if it has two (not necessarily disjoint) sets each of size roughly $n/3$ with few edges in between.
However, for any $\alpha >0$, a graph $G$ with $\delta(G) \ge (2/3+\alpha)n$ cannot be extremal, thus we can use all lemmas from~\cite{square_ham_noreg} dealing with the non-extremal case. We summarise the statements we use from~\cite{square_ham_noreg} as follows.

\begin{theorem}[Lemma~$3$, Lemma~$5$, and Theorem~$1$ in~\cite{square_ham_noreg}]
\label{thm:posa}
For any $\alpha>0$ there exists $n_0$ such that for $n \ge n_0$ the following holds for any $n$-vertex graph with minimum degree $\delta(G) \ge (2/3 +\alpha)n$.
\begin{enumerate}[label=\upshape(P\arabic*)]
\item \label{thm:posa:con} For any two disjoint ordered edges $(a,b)$ and $(c,d)$ there is a square of a path of length at most $10 \alpha^{-4}$, with end-tuples $(a,b)$ and $(c,d)$.
\item \label{thm:posa:abs} There exists the square of a path $P$ of length at most $\alpha^9 n$ such that for every subset $L \subseteq V(G) \setminus V(P)$ there exists a square of a path $P_L$ with $V(P_L)=V(P) \cup L$ that has the same end-tuples as $P$.
\item \label{thm:posa:threshold} There exists the square of a Hamilton cycle in $G$.
\end{enumerate}
\end{theorem}

Finally the property $\textbf{Fac}$ for $\cA$ follows as a special case of a theorem in~\cite{alp_yani_richard}.

\begin{theorem}[Theorem~$1.3$ in~\cite{alp_yani_richard}] 
    \label{thm:factor_rb}
    For any integer $r\ge 1$ and any $\alpha>0$, there exists $n_0 \in \mathbb{N}$ such that for $n \ge n_0$ the following holds.
    Any graph collection $\fG$ on $[n]$ with $\delta(\fG) \ge \left(\tfrac{r}{r+1} + \alpha\right)n$ contains a transversal copy of a $K_{r+1}$-factor.
\end{theorem}

We are now ready to give a full proof of Theorem~\ref{thm:posa_rb}.

\begin{proof}[Proof of Theorem~\ref{thm:posa_rb}]
Let $\alpha>0$ and $\mathcal{A}$ be the $2$-link of order $3$ and uniformity $2$ coming from an arbitrary ordering of $K_3$.
Note that an $\mathcal{A}$-chain is the square of a path (see Figure~\ref{fig:square}) and an $\mathcal{A}$-cycle is the square of a cycle.

The minimum degree threshold for a Hamilton $\mathcal{A}$-cycle is $\delta=\delta(\mathcal{A},1)=2/3$ by~\ref{thm:posa:threshold} of Theorem~\ref{thm:posa}.
Let $n_0$ be large enough for Theorem~\ref{thm:posa} and~\ref{thm:factor_rb} to hold.
Then $\mathcal{A}$ has property \textbf{Ab} with $\tau=\alpha^9$ and $\eta=\alpha^{20}$ by~\ref{thm:posa:abs} of Theorem~\ref{thm:posa}, and it has property \textbf{Con} with $C=10\alpha^{-4}$ by~\ref{thm:posa:con} of Theorem~\ref{thm:posa}.
Moreover, $\mathcal{A}$ has property \textbf{Fac} with $\beta_0=1$ by Theorem~\ref{thm:factor_rb} (with $r=2$).
Therefore $\cA$ is $(\delta,1)$-good.

Now let $\fG$ be a graph collection on $[n]$ with $\delta(\fG) \ge (2/3 + \alpha)n$ with $n \ge n_0$ .
Then, by Theorem~\ref{thm:main}, there exists a rainbow Hamilton $\mathcal{A}$-cycle in $\fG$, i.e.~a transversal copy of the square of a Hamilton cycle, as desired.
\end{proof}

To obtain part~\ref{thm:app_HC2} of Theorem~\ref{thm:applications} for $r>2$ we can proceed exactly as for $r=2$.
However, the statements in~\cite{pavez2021towards} do not readily match our setup as those given in Theorem~\ref{thm:posa}.
Nevertheless, Lemma~4.3 in~\cite{pavez2021towards} implies property \textbf{Con} and it is straightforward to check that together with Lemma~7.2 in~\cite{pavez2021towards} this also gives property \textbf{Ab}.
Indeed, Lemma~7.2 in~\cite{pavez2021towards} states that if $G$ is a graph with $\delta(G) \ge (r/(r+1)+\alpha)n$ and $n$ is large enough, then there is a small set of pairwise vertex-disjoint $r$-th powers of short paths, such that every vertex of $G$ can be absorbed into many of them (into the $r$-th power of a path).
These paths can then be connected into the $r$-th power of a single path to fulfil property \textbf{Ab} (c.f.~Step~1 of the proof of Theorem~1.1 in~\cite{pavez2021towards} for more details).
As property \textbf{Fac} still holds by Theorem~\ref{thm:factor_rb}, we have that $\cA$ is $(\delta,1)$-good, for $\cA$ being the $(r-1)$-link of order $r$ and uniformity $2$ coming from an arbitrary ordering of $K_r$ and $\delta=\delta(\cA,1)=r/(r+1)$.
The result follows by Theorem~\ref{thm:main}.

\subsection{Hamilton $\ell$-cycles in $k$-uniform hypergraphs}

The statements in~\ref{thm:app_cycle} of Theorem~\ref{thm:applications} state $d$-colour-blindness of the family $\cF$ of $k$-uniform Hamilton $\ell$-cycles, for various ranges of $d$, $k$, and $\ell$.
Note that an $\ell$-cycle in a $k$-uniform hypergraph is an $\cA$-cycle, with $\cA$ being the $\ell$-link of order $k$ and uniformity $k$ consisting of a single edge (see Figure~\ref{fig:edge}).
The result will follow from our main theorem, once we will have shown that such $\cA$ is $(\delta,d)$-good, with $\delta$ being the uncoloured minimum degree threshold of the considered family $\cF$.

We start by observing that, since $\cA$ consists of a single edge, Observation~\ref{obs:fac} guarantees that $\cA$ satisfies property $\fac$ for any $k \ge 3$, and $1 \le \ell, d \le k$.
The properties $\ab$ and $\con$ can be derived from the absorption-style proof of the uncoloured minimum degree threshold for $\cF$. 
We summarise the precise reference for each property and each case of the statements in~\ref{thm:app_cycle} of Theorem~\ref{thm:applications} in Table~\ref{fig:table_applications}.
Although some of these lemmas are not stated in the same exact form of the corresponding property, it is always straightforward to derive the properties from the lemmas. 

Nevertheless, we clarify a few points.
Firstly, we consider the second row of Table~\ref{fig:table_applications}, where $1 \le \ell < k/2$ and $d=k-1$.
Lemma $6$ in~\cite{codegree_l_kunif} states that for every integer $k \ge 2$ and every pair of real numbers $d, \eps > 0$, there exists an $n_0$ such that for every $k$-uniform hypergraph $\cH$ on $n$ vertices with $\delta(\cH) \ge dn$ the following holds.
There is a set $R$ of size at most $\eps n$ such that each set of $k-1$ vertices has degree at least $d \eps n/2$ into $R$.
This implies property $\con$ with $c=3$. 
Indeed, given two edges $\cS$ and $\mathcal{T}$ in $\cH$, since $2\ell \le d$ and using the property of $R$, we can find an additional edge of $\cH$ and connect $\cS$ and $\mathcal{T}$ into an $\ell$-path of length $3$.
Secondly, we consider the last row of
Table~\ref{fig:table_applications}, where $\ell=k/2$ and $k/2 < d \le k-1$ with $k$ even.
The authors of~\cite{han2022minimum} prove an exact uncoloured minimum degree threshold, by distinguishing between an extremal and a non-extremal case. 
It is easy to see that any hypergraph $\cH$ with $\delta_d(\cH) \ge (\delta_{\cF, d}+\alpha) n^{k-d}$ is non-extremal, and thus we can use all lemmas from~\cite{han2022minimum} dealing with the
non-extremal case.

Statements in~\ref{thm:app_cycle} of Theorem~\ref{thm:applications} can now be proved using the same arguments as in the proof of Theorem~\ref{thm:posa_rb}.

\begin{table}[htbp]
	\begin{center}
	\resizebox{\textwidth}{!}{  
	\begin{tabular}{|c||c|c|c|c|} 
	    \hline
        Family $\cF$ & Reference for $\delta_{\cF,d}$ & Property $\ab$ & Property $\con$ & Property $\fac$ \\
        \hline
        \hline
        & Bu\ss, H\`an, & \multirow{2}{*}{Lemma $7$ in~\cite{vertex_loose_3unif}} & \multirow{2}{*}{Lemma $5$ in~\cite{vertex_loose_3unif}} & \multirow{11}{*}{Observation~\ref{obs:fac}}\\
        $1 < \ell < k/2$ & and Schacht~\cite{vertex_loose_3unif} & & & \\
        \cline{2-4}
        and $d=k-2$ & de Bastos, Mota, Schacht, & \multirow{2}{*}{Lemma $7$ in~\cite{de2017loose}} & \multirow{2}{*}{Lemma $5$ in~\cite{de2017loose}} &  \\
         & Schnitzer, and Schulenburg~\cite{de2017loose} & & & \\
        \cline{1-4}
        $1 \le \ell < k/2$ & \multirow{2}{*}{H\`an and Schacht~\cite{codegree_l_kunif}} & \multirow{2}{*}{Lemma $5$ in~\cite{codegree_l_kunif}} & \multirow{2}{*}{Lemma $6$ in~\cite{codegree_l_kunif}} & \\
        and $d=k-1$ & & & & \\
        \cline{1-4}
        $\ell = k-1$ & R\"{o}dl, Ruci\'{n}ski, & \multirow{2}{*}{Lemma $2.1$ in ~\cite{codegree_tight}} & \multirow{2}{*}{Lemma $2.4$ in ~\cite{codegree_tight}} & \\
        and $d=k-1$ & and Szemer\'{e}di~\cite{codegree_tight} & & &\\        \cline{1-4}
        $\ell=k/2$ & \multirow{3}{*}{H\`an, Han, and Zhao~\cite{han2022minimum}} & \multirow{3}{*}{Lemma $2.3$ in~\cite{han2022minimum}} & \multirow{3}{*}{Lemma $2.5$ in~\cite{han2022minimum}} & \\
        and $k/2 < d \le k-1$, & & & &\\
        with $k$ even & & & & \\
        \hline
	\end{tabular}	
	}
	\end{center}
	\caption{\label{fig:table_applications}References for the properties $\ab$, $\con$, and $\fac$ for the families in the statement~\ref{thm:app_cycle} of Theorem~\ref{thm:applications}.
	The first row is split into two, as~\cite{vertex_loose_3unif} deals with the case $k=3$ and~\cite{de2017loose} deals with the case $k \ge 4$.}
\end{table}
Of course, Table~\ref{fig:table_applications} is not an exhaustive list of all Dirac-type results proven via the absorption method, rather it is only a small sample.

\section{Concluding remarks}
\label{sec:concluding_remarks}

We remark that any transversal embedding problem for an $n$-vertex $m$-edge $k$-uniform hypergraph is equivalent to a non-rainbow embedding problem in a $(k+1)$-uniform $2$-partite hypergraph with parts of size $m$ and $n$. Indeed, each vertex in the class of size $m$ represents one of the colours and forms a $(k+1)$-edge together with each edge of this colour. This setup is somewhat unnatural and has not been studied explicitly, but this perspective was very helpful in many of the results for rainbow structures~\cite{cheng2021hamilton,Lu_codegree_rb-matching_tight,lu2022rainbow,rainbow_matching}. \par We now highlight further directions of research that we find to be of most interest.

\subsection{Vertex degree for tight Hamilton cycles}
\label{sec:open_problem_6.1}
Theorem \ref{thm:applications}\ref{thm:app_cycle} proves $d$-colour-blindness of the family of $k$-uniform Hamilton $\ell$-cycles, for various ranges of $d$, $k$, and $\ell$.
However, there is a well-known (uncoloured) Dirac-type result whose rainbow version is missing there: the vertex minimum degree for tight Hamilton cycles in $3$-uniform hypergraphs, corresponding to $d=1$, $k=3$ and $\ell=k-1$.

The proof of the minimum vertex degree threshold for this family is due to Reiher, Rödl, Ruciński, Schacht, and Szemerédi~\cite{reiher2019minimum}, and it uses the absorption method, making this family an ideal candidate for our main theorem.
However, it turns out that we cannot hope for the property \textbf{Con} to hold in this range of the parameters (see Section $2.1$ in~\cite{reiher2019minimum} for a discussion). Due to this additional complication, it would be an interesting challenge to obtain a transversal generalisation of the result in~\cite{reiher2019minimum}.

\subsection{Exact results and stability}

For Hamilton cycles in graphs, the exact rainbow minimum degree threshold is known~\cite{joos2020rainbow}, and the family of Hamilton cycles is \emph{exactly} colour-blind, meaning that an error-term as in Definition 1.2 is not required.
It is natural to ask whether exact results also hold for other structures in the rainbow setup and if a general statement similar to our Theorem~\ref{thm:main} can be proved.
For example, improving on the statement~\ref{thm:app_HC2} of Theorem~\ref{thm:applications}, it would be very interesting to show if $\delta(\fG) \ge rn/(r+1)$ is already sufficient for a transversal copy of the $r$-th power of a Hamilton cycle in graphs.
Note that already for a rainbow $K_r$-factor it is not known whether $\delta(\fG) \ge rn/(r+1)$ suffices.
Moreover, resolving this for the $r$-th power of a Hamilton cycles does not immediately imply the analogous result for a $K_r$-factors, even though the former contains the latter, because of the different number of colours needed for a rainbow embedding.
A similar observation is true for tight Hamilton cycles and perfect matchings in hypergraphs. We remark that Lu, Wang, and Yu~\cite{Lu_codegree_rb-matching_tight} showed that the family of $k$-uniform perfect matchings is \emph{exactly} $(k-1)$-colour blind, proving that the rainbow minimum co-degree threshold essentially is $n/2$. Improving on one of the statements in~\ref{thm:app_cycle} of Theorem~\ref{thm:applications}, we can ask if the same condition $\delta_{k-1}(\fH)\ge n/2$ is sufficient for a transversal copy of a tight Hamilton cycle. 

In the non-rainbow setup, exact results can typically be obtained by considering an extremal and non-extremal case separately, where the latter often gives stability, i.e.~even a smaller minimum degree condition is sufficient if the graph is far from any extremal construction.
For graphs there are also more ad-hoc arguments that also work in the rainbow setup, e.g.~\cite{joos2020rainbow}.
For perfect matchings in collections of $k$-uniform hypergraphs, Lu, Wang, and Yu~\cite{Lu_codegree_rb-matching_tight} transform the problem into a matching problem in a $(k+1)$-uniform $2$-partite hypergraph as explained above.
Their arguments uses absorbers and distinguishes between an extremal and non-extremal case.
Roughly speaking, they say that a hypergraph collection is extremal if essentially all of them are close to one of the extremal graphs for the uncoloured problem.
Working with a similar notion for an extremal collection, it would be interesting to prove an exact version for any of our results, as many of them hold in the uncoloured setup, e.g.~\cite{bastos2018loose,han2022minimum,han2015minimum2,han2015minimum,square_Ham_cycle,rodl2011dirac}.

It seems to be too much to hope for a general theorem that covers all of these applications, because of the different extremal constructions in each case.
But we remark that, besides the properties \textbf{Ab}, \textbf{Con}, and \textbf{Fac}, the additional $(\alpha n^{k-d})$-term for the minimum $d$-degree in our theorem is only necessary for the two applications of Lemma~\ref{lem:rainbow_paths_tiling}.
Therefore, a major step would be a variant of this theorem (for specific $\cA$) which is applicable with a lower minimum $d$-degree, under the assumption that the hypergraph collection is not extremal.
However, Hamilton cycles give new complications and in this setup it is harder to make a direct use of the results from the non-rainbow case, which was the main scope of this paper.

\subsection{Other potential applications}
There are many more structures which can be represented as $\mathcal{A}$-cycles, e.g.~copies of $C_4$ glued as depicted in Figure~\ref{fig:pillar}. In the case of graphs, any $\mathcal{A}$-link forms an $\mathcal{A}$-cycle with bounded maximum degree and bounded bandwidth. Thus, the bandwidth theorem by Böttcher, Schacht and Taraz~\cite{bandwidthThm} immediately gives minimum degree thresholds for the existence of a Hamilton $\mathcal{A}$-cycle, but their proof relies on different techniques than we require for the application of Theorem~\ref{thm:main}. Hence, one would need to prove properties \textbf{Ab}, \textbf{Con}, and \textbf{Fac} for such structures in order to obtain the corresponding rainbow result using our method.
More generally, a rainbow version of the bandwidth theorem would be very interesting.
Note that the bandwidth theorem is not optimal for many graphs, so the minimum degree conditions for the containment of Hamilton $\mathcal{A}$-cycles is an interesting problem even in the non-rainbow setup, and is related to embedding factors and the critical chromatic number~\cite{kuhnosthus}.

\begin{remark}
    After the manuscript became available online,  the problem in Section~\ref{sec:open_problem_6.1} was solved by Tang, Wang, Wang, and Yan~\cite{tang2023rainbow}, and a rainbow bandwidth theorem for graph transversal was proved by  Chakraborti, Im, Kim, and Liu~\cite{chakraborti2023bandwidth}.
\end{remark}

\section*{Acknowledgements}
We thank Julia B\"ottcher for suggesting us to investigate if the family of the squares of Hamilton cycles is colour-blind and Richard Lang for bringing~\cite{pavez2021towards} to our attention. We thank Richard Montgomery for valuable comments regarding the presentation of the paper. We also thank Anusch Taraz and Dennis Clemens for organising an online workshop in March 2021 where this project was initiated.

\bibliographystyle{plain}
\bibliography{references}

\end{document}